\documentclass[a4paper,11pt]{article}
\usepackage{amssymb,amsmath,amsthm, amsfonts}
\usepackage{a4wide}
\usepackage{color}

\newcommand{\blue}[1]{\textcolor{blue}{#1}}
\newcommand{\delHG}[1]{\blue{\ \ldots\ }}

\newcommand{\at}[1]{}

\newcommand{\Gr}{{\rm gph\,}}

\newcommand{\co}{{\rm conv\,}}

\newcommand{\cl}{{\rm cl\,}}
\newcommand{\ri}{{\rm ri\,}}

\newcommand{\argmax}{\mathop{\rm arg\,max\,}\limits}
\newcommand{\xb}{\bar x}
\newcommand{\yb}{\bar y}
\newcommand{\zb}{\bar z}
\newcommand{\wb}{\bar w}
\newcommand{\yba}{\yb^\ast{}}
\newcommand{\vb}{\bar v}
\newcommand{\ub}{\bar u}
\newcommand{\vba}{\vb^\ast{}}
\newcommand{\zba}{\zb^\ast{}}

\newcommand{\lb}{{\bar\lambda}}

\newcommand{\I}{{\cal I}}
\newcommand{\J}{{\cal J}}
\newcommand{\A}{{\cal A}}
\newcommand{\F}{{\cal F}}
\newcommand{\Lsp}{{\cal L}}

\newcommand{\KbG}{{\bar K_\Gamma}}
\newcommand{\Tlin}{T^{\rm lin}}
\newcommand{\TlinO}{\Tlin_{P,D}}

\newcommand{\Lb}{\bar\Lambda}
\newcommand{\Lbv}{\bar\Lambda(v)}
\newcommand{\R}{\mathbb{R}}

\newcommand{\norm}[1]{\|#1\|}

\newcommand{\dist}[1]{{\rm d}(#1)}
\newcommand{\B}{{\cal B}}

\newcommand{\E}{{\cal E}}

\newcommand{\K}{{\cal K}}

\newcommand{\mv}{\,\vert\, }

\newcommand{\skalp}[1]{\langle #1\rangle}

\newcommand{\range}{{\rm Range\,}}

\newcommand{\Span}{\mathop{\rm span\,}\limits}

\newcommand{\vek}[1]{\left(\begin{array}{c}#1\end{array}\right)}

\newtheorem{definition}{Definition}
\newtheorem{theorem}{Theorem}
\newtheorem{lemma}{Lemma}
\newtheorem{example}{Example}
\newtheorem{corollary}{Corollary}
\newtheorem{proposition}{Proposition}
\newtheorem{remark}{Remark}
\newtheorem{assumption}{Assumption}
\begin{document}
\title{New sharp necessary optimality conditions for
 mathematical programs with equilibrium
constraints}
\author{Helmut Gfrerer\thanks{Institute of Computational Mathematics, Johannes Kepler University Linz, A-4040 Linz, Austria, e-mail:
helmut.gfrerer@jku.at. {The research of this author was partially supported by the Austrian Science Fund (FWF) under grant P29190-N32.}} \and Jane J. Ye\thanks{Department of Mathematics
and Statistics, University of Victoria, Victoria, B.C., Canada V8W 2Y2, e-mail: janeye@uvic.ca. The research of this author was partially
supported by NSERC.}}
\date{}
\maketitle
\begin{abstract}
In this paper, we study the mathematical program with equilibrium constraints (MPEC) formulated as a mathematical program with  a parametric generalized equation involving {the} regular normal cone. We derive a new necessary optimality  condition which  is sharper than the usual M-stationary condition
and is applicable even when no   constraint qualifications hold for  the corresponding mathematical program with complementarity constraints (MPCC)  reformulation.

\vskip 10 true pt


\vskip 10 true pt


\end{abstract}


\maketitle

\section{Introduction}
In this paper we consider the {\em mathematical program with equilibrium constraints} (MPEC) of the form
\begin{eqnarray}
\label{EqMPEC}  {\rm (MPEC)}\qquad\min_{x,y}&&F(x,y)\\
\nonumber  \mbox{s.t.}&&0\in \phi(x,y)+\widehat N_\Gamma(y),\\
\nonumber  &&G(x,y)\leq 0,
\end{eqnarray}
where  $\Gamma:= \{y\mv g(y)\leq 0\}$ and $\widehat N_\Gamma(y)$ denotes the so-called {\em regular normal cone} to the set $\Gamma$ at $y$  (see Definition \ref{DefVarGeometry}). Here we assume that  $F:\R^n\times\R^m\to\R$, $\phi:\R^n\times\R^m\to\R^m$, $G:\R^n\times\R^m\to\R^p$ are  continuously differentiable and $g:\R^m\to\R^q$ is twice continuously differentiable.

In the case where $\Gamma$
 is convex, $\widehat N_\Gamma(y)=N_\Gamma(y)$ is the normal cone in the sense of convex analysis and (MPEC) is equivalent to the {\em mathematical program with variational inequality constraints} (MPVIC) which arised in many applications from engineering and economics;  see e.g. \cite{Luo-Pang-Ralph,Out-Koc-Zowe} and the references within.

Up to now a common approach for handling (MPEC) is its reformulation as a {\em mathematical program with complementarity constraints} (MPCC). If at a point $y\in\Gamma$ a certain constraint qualification is fulfilled,  then by the Karush-Kuhn-Tucker (KKT)  condition,
\[0\in \phi(x,y)+\widehat N_\Gamma(y)\ \Longleftrightarrow \exists \lambda : 0= \phi(x,y)+\nabla g(y)^T \lambda,
\ \   0\leq -g(y) \perp \lambda \geq 0.\]
This observation yields the program
\begin{eqnarray*}
\mbox{(MPCC)} \qquad\min_{x,y,\lambda} &&F(x,y)\nonumber\\
 \mbox{s.t.}&&0= \phi(x,y)+\nabla g(y)^T \lambda, \\
 && 0\leq -g(y) \perp \lambda \geq 0,\\
 &&G(x,y)\leq 0,
\end{eqnarray*}
which has been considered extensively in the literature during the last three decades.
However, since in (MPCC) the minimization is over the original variable $x,y$ as well as the multiplier $\lambda$, it is not equivalent to the original problem (MPEC) in general, cf.\cite{Dam-Dut}. Moreover as discussed in \cite{Ad-Hen-Out,GfrYe16a}, there are many difficulties involved in using reformulation (MPCC) and therefore it is favorable to consider (MPEC) instead of (MPCC).

In  Ye and Ye \cite{YeYe},  the  calmness/pseudo upper-Lipschitz continuity  of the perturbed feasible  mapping  of (MPEC)  has been proposed and proven to be a constraint qualification for the Mordukhovich (M-) stationarity to hold at a minimizer. The calmness of  the perturbed feasible  mapping of (MPEC) is known to be equivalent to the subregularity of the set-valued map
\begin{equation*}
M_{\rm MPEC}{(x,y)}:=\vek{\phi(x,y)+\widehat N_\Gamma(y)\\G(x,y)-\R^p_-},
\end{equation*}
which we refer to be the {\em metric subregularity constraint qualification} (MSCQ).
In \cite[Theorem 5]{GfrYe16a} (Theorem \ref{ThSuffCondMS_FOGE} in this paper), a concrete sufficient condition in terms of the problem data is provided for MSCQ. Continuing the work in  \cite{GfrYe16a}, in  this paper we aim at developing a new sharp necessary optimality condition for problem (MPEC).

Recently  \cite{Gfr18} (see Theorem \ref{ThLinM_StatBasic} in this paper) derived a new necessary optimality condition for an optimization problem with a set-constraint  in the form of $P(z)\in D$ where $P$ is continuously differentiable and $D$ is a closed set. The new optimality condition is  derived in terms of the so-called  {\em linearized Mordukhovich (M-) stationary condition} which is stronger than the usual M-stationarity condition.

It is easy to see that the constraint of (MPEC) can be rewritten in the form
$$P(x,y):=\left(\begin{array}{c}(y,-\phi(x,y))\\G(x,y)\end{array}\right)\in D:=\Gr \widehat N_\Gamma\times \R^p_-$$
and hence (MPEC) can be treated as an optimization problem with the above set-constraint. In \cite[Theorem 5]{Gfr18},  under some constraint qualifications on the lower level constraint $g(y)\leq 0$,  which can be guaranteed to hold under the constant rank constraint qualification (CRCQ), the   linearized M-stationary condition for (MPEC) is derived. In this paper, we drop this constraint qualification and  we derive the linearized M-necessary optimality condition under the so-called 2-nondegeneracy condition on $g(y)\leq 0$.

We organize our paper as follows. Section 2 contains the preliminaries from variational geometry and variational analysis. In Section 3, we recall the linearized M-
optimality conditions for the optimization problem with a set-constraint.
In Section 4, we discuss constraint qualifications for (MPEC).  In Section 5, under the 2-nondegeneracy condition, we derive formula for regular normal cones to tangent directions that will be used in applying the necessary optimality condition from Section 3. Finally in Section 6, we reformulate (MPEC) in the form of an optimization problem with a set-constraint and apply the necessary optimality condition from Section 3.

 The following {notation} will be used throughout the paper. We denote by $\B_{\R^q}$ the closed unit ball in $\R^q$ while when 
no confusion arises we denote it by $ \B$. By $\B(\zb; r)$ we denote the closed ball centered at $\zb$ with radius $r$.
 For a matrix
$A$, we denote by $A^T$ its transpose. The inner product of two vectors $x, y$ is denoted by
$x^T y$ or $\langle x,y\rangle$ and by $x\perp y$ we mean $\langle x, y\rangle =0$. For $\Omega \subseteq \R^d$ and $z \in \R^d$, we denote by $\dist{z, \Omega}$ the distance from $z $ to $\Omega$. The polar cone of a set $\Omega$ is
$\Omega^\circ:=\{x|x^Tv\leq 0 \ \forall v\in \Omega\}$ and $\Omega^\perp$ denotes the orthogonal complement to $\Omega$. For a set $\Omega$, we denote by $\co \Omega$ and $\cl\Omega$ the convex hull
and the closure  of $\Omega$,  respectively. For a function $f:\R^d \rightarrow \R$, we denote by $\nabla f(\bar z)$ the gradient vector of $f$ at $\bar z$ and  $\nabla^2 f(\bar z)$ the Hessian matrix of $f$ at $\bar z$.
For a  mapping $P:\R^d\rightarrow \R^s$ with $s>1$, we denote by $\nabla P(z)$ the Jacobian matrix of $P$ at $z$ and  for any given $w,v\in \R^d$, $w^T \nabla P(\bar z) v$ is the vector in  $\R^s$ with the $i$th component equal to $w^T \nabla^2 P_i(\bar z) v, i=1,\dots, s$.  Let $M:\R^d\rightrightarrows\R^s$ be an arbitrary set-valued mapping. We denote its graph by $ {\rm gph}M:=\{(z,w)| w\in M(z)\}.$ $o:\R_+\rightarrow \R$ denotes a function with the property that $o(\lambda)/\lambda\rightarrow 0$ when $\lambda\downarrow 0$.
\section{Preliminaries from variational geometry and variational analysis}
In this section, we gather some preliminaries and preliminary results in variational analysis that will be needed in the paper. The reader may find more details in the monographs \cite{Clarke,Mor,RoWe98} and  in the papers we refer to.

\begin{definition}[Tangent cone and normal cone]\label{DefVarGeometry}
Given a set
$\Omega\subseteq \mathbb R^d$ and a point $\bar z\in\Omega$,
the (Bouligand-Severi) {\em tangent/contingent cone} to $\Omega$
at $\bar z$ is a closed cone defined by
\begin{equation*}
T_\Omega(\bar z)
:=\limsup_{t\downarrow 0}\frac{\Omega-\bar z}{t}
=\Big\{u\in\mathbb R^d\Big|\;\exists\,t_k\downarrow
0,\;u_k\to u\;\mbox{ with }\;\bar z+t_k u_k\in\Omega ~\forall ~ k\}.
\end{equation*}
The (Fr\'{e}chet) {\em regular normal cone} and the (Mordukhovich) {\em limiting/basic normal cone} to $\Omega$ at $\bar
z\in\Omega$ are  closed cones defined by
\begin{eqnarray}
&& \widehat N_\Omega(\bar z):=(T_\Omega(\bar z))^\circ\nonumber\\
\mbox{and }  &&
N_\Omega(\bar z):=\left \{z^\ast \mv \exists z_{k}\stackrel{\Omega}{\to}\zb \mbox{ and } z^\ast_k\rightarrow z^\ast \mbox{ such that } z^\ast_{k}\in \widehat{N}_{\Omega}(z_k) \  \forall k \right \},
\nonumber
\end{eqnarray}
respectively.
\end{definition}
When  the set $\Omega$ is convex, the tangent/contingent cone and the regular/limiting normal cone reduce to
the classical tangent cone and normal cone of convex analysis,  respectively.

\begin{definition}[Metric regularity and subregularity]
  	 Let $M:\R^d\rightrightarrows\R^s$ be a set-valued mapping and let $(\zb,\wb)\in\Gr M$.
  \begin{enumerate}
  \item[(i)]  We say that $M$ is {\em metrically subregular} at $(\zb,\wb)$ if there exist a neighborhood $Z$ of $\zb$ and a positive number  $\kappa>0$ such that
  \begin{equation*}\dist{z,M^{-1}(\wb)}\leq\kappa\dist{\wb,M(z)}\ \; \forall z\in Z.
  \end{equation*}
  \item[(ii)]   We say that $M$ is {\em metrically regular around} $(\zb,\wb)$ if there exist  neighborhoods $Z$ of $\zb$, $W$ of $\wb$ and a positive number  $\kappa>0$ such that
  \begin{equation*}\dist{z,M^{-1}(w)}\leq\kappa\dist{w,M(z)}\ \; \forall (w,z)\in W\times Z.
  \end{equation*}
  \end{enumerate}
\end{definition}
It is well-known that metric subregularity of $M$ at $(\zb,\wb)$ is equivalent with the property of {\em calmness} of the inverse mapping $M^{-1}$ at $(\wb,\zb)$, cf. \cite{DoRo04}, whereas metric regularity of $M$ around $(\zb,\wb)$ is equivalent with the {\em Aubin property} of the inverse mapping $M^{-1}$ around $(\wb,\zb)$.  It follows immediately from the definition that metric regularity of $M$ around $(\bar z,\bar w)$ implies metric subregularity. Further,
metric subregularity of $M$ at $(\zb,\wb)$ is equivalent with metric subregularity of the mapping $z\to (z,\bar w)-\Gr M$ at $(\zb,(0,0))$, cf. \cite[Proposition 3]{GfrYe16a}.

Metric regularity can be verified via the so-called {\em Mordukhovich-criterion}. We give here only reference to a special case which is used  in the sequel.
\begin{theorem}[Mordukhovich criterion](cf. {\cite[Example 9.44]{RoWe98}})
Let $P:\R^d \to\R^s$ be continuously differentiable, let $D \subseteq  \R^s$ be closed and let $P(\zb)\in D$. Then the mapping $z\rightrightarrows P(z)-D$ is metrically regular around $(\zb,0)$ if and only if
\begin{equation}\label{EqMoCrit}
\nabla P(\zb)^Tw^*=0,\ w^*\in N_D\big(P(\zb)\big)\ \Longrightarrow\ w^*=0.
\end{equation}
\end{theorem}
For verifying the property of metric subregularity there are some  sufficient conditions  known, see e.g. \cite{Gfr11, Gfr13a, Gfr13b,Gfr14a, GfrKl16}.

\begin{definition}[Critical cone] For a closed set $\Omega \subseteq \R^d$, a point $z\in\Omega$ and a regular normal $z^*\in\widehat N_\Omega(z)$ we denote by
\[\K_\Omega(z,z^*):=T_\Omega(z)\cap [z^*]^\perp\]
the {\em critical cone} to $\Omega$ at $(z,z^*)$
\end{definition}
In this paper polyhedrality will play an important role.
\begin{definition}[Polyhedrality]
\begin{enumerate}
  \item Let $C\subseteq  \R^d$.
  \begin{enumerate}
   \item We say that $C$ is {\em convex polyhedral}, if it can be written as the intersection of finitely many halfspaces, i.e. there are elements $(a_i,\alpha_i)\in\R^d\times\R$, $i=1,\ldots,p$ such that $C=\{z\mv \skalp{a_i,z}\leq\alpha_i,\ i=1,\ldots,p\}$.
   \item $C$ is said to be  {\em polyhedral}, if it is the union of finitely many convex polyhedral sets.
    \item Given a point $c\in C$, we say that $C$ is {\em locally polyhedral near $c$} if there is a neighborhood $W$ of $c$ and a polyhedral set $\tilde C$ such that $C\cap W= \tilde C\cap W$.
    \end{enumerate}
    \item A mapping $M:\R^d\rightrightarrows \R^s$ is called {\em polyhedral}, if its graph $\Gr M$ is a polyhedral set.
\end{enumerate}
\end{definition}

\begin{lemma}\label{Normalcones}
  Let $\Omega\subseteq\R^d$ be locally polyhedral near some point $\zb\in \Omega$. Then
  \begin{equation*}
    N_\Omega(\zb)=\bigcup_{w\in T_\Omega(\zb)}\widehat N_{T_\Omega(\zb)}(w).
  \end{equation*}
\end{lemma}
\begin{proof}
  Follows from \cite[Lemma 2.2]{Gfr14a}.
\end{proof}
We recall some properties of closed cones.
\begin{proposition}\label{closedcone}  Let $K$ be a closed cone in $\mathbb{R}^d$. Then
$T_K(0)=K$.
\end{proposition}
A consequence of the above property is that for any $y\in D$,
\begin{equation}\label{normaltangent}
\widehat N_{T_D(y)}(0)=\big ( T_{T_D(y)}(0)\big  )^\circ=\big (T_D(y)\big )^\circ=\widehat N_D(y).\end{equation}

The following  rules for calculating polar cones will be useful.
\begin{proposition}\label{calpolar}
(\cite[Corollary 16.4.2]{Ro70})\label{Prop1}  Let $A, B$ be nonempty  convex cones in $\mathbb{R}^d$. Then
$(A+B)^\circ=A^\circ\cap B^\circ$. 
{If both $A$ and $B$ are closed then  $(A\cap B)^\circ=\cl(A^\circ +B^\circ)$.}\end{proposition}
{In this paper we use Proposition \ref{calpolar} solely in situations, when the closure operation in the last formula can be omitted, namely, when either $A^\circ+B^\circ$ is a subspace or when both $A$ and $B$ are convex polyhedral cones, cf. \cite[Corollaries 19.2.2, 19.3.2]{Ro70}.}

In the following proposition we collect some important facts about the normal cone mapping to a convex polyhedral set $C$, which can be extracted from \cite{DoRo96}.
\begin{proposition}[Normal cone to a convex polyhedral set]\label{PropNormalConeMapping} Let $C\subseteq \R^d$ be a polyhedral convex set and $(\zb,\zba)\in\Gr N_C$. Then for all $z\in C$ sufficiently close to $\zb$ we have
\begin{equation}\label{EqIncLTanNormalConePoly}T_C(z) \supseteq  T_C(\zb),\quad N_C(z)\subseteq N_C(\zb).\end{equation}
Further, there exists a neighborhood $W$ of $(\zb,\zba)$ such that
\begin{equation*}
\Gr N_C \cap W=((\zb,\zba)+\Gr N_{\K_C(\zb,\zba)})\cap W.
\end{equation*}
In particular we have
\begin{equation}\label{EqTanNormalConePolyhedral}
  T_{\Gr N_C}(\zb,\zba)=\Gr N_{\K_C(\zb,\zba)}.
\end{equation}
Further,
\begin{equation}\label{EqNormalConePolyhedral}
  \widehat N_{\Gr N_C}(\zb,\zba)=(\K_C(\zb,\zba))^\circ\times \K_C(\zb,\zba)
\end{equation}
and the limiting normal cone $N_{\Gr N_C}(\zb,\zba)$ is the union of all sets of the form
\begin{equation*}(F_1-F_2)^\circ\times(F_1-F_2)\end{equation*}
where $F_2\subseteq  F_1$ are faces of $\K_C(\zb,\zba)$.
\end{proposition}

\begin{definition}[Recession Cone](\cite[page 61]{Ro70})
Let $C\subseteq  \mathbb{R}^d$ be a closed convex set. The recession cone of $C$ is a closed convex cone defined as
$
0^+C := \{ y\in \mathbb{R}^d | x+\lambda y \in C \quad \forall \lambda \geq 0, x\in C\}.
$\end{definition}

\begin{definition}[Generalized lineality space] Given an arbitrary set $C\subseteq \R^d$, we call a subspace $L$ the generalized lineality space of $C$ and denote it by  $\Lsp(C)$ provided that it is   the largest subspace $L\subseteq \R^d$ such that
$C+L\subseteq C.$
\end{definition}
\noindent Note that $\Lsp(C)$ is well defined because for two subspaces $L_1,L_2$ fulfilling $C+L_i\subseteq C$, $i=1,2$ we have
$C+L_1+L_2=(C+L_1)+L_2\subseteq C+L_2\subseteq C$ and hence we can always find a largest subspace satisfying $C+L\subseteq C$ since the dimension $\R^d$ is finite.
 Note that since $0$ is in every subspace  we have $C+L\supseteq C$ and thus $C+\Lsp(C)=C$.  In the case where $C$ is a convex set, the generalized lineality space reduces to the lineality space as defined in \cite[page 65]{Ro70} and can be calculated as $\Lsp(C)=(-0^+C)\cap 0^+C$. In the case where $C$ is a convex cone,  the lineality space  of $C$ is the largest subspace contained in $C$ and can be calculated as $\Lsp(C)=(-C)\cap C$.

By definition  of the generalized lineality space and the tangent cone, it is easy to verify that for every $\zb\in C$ we have
\begin{equation}\Lsp(C)\subseteq\Lsp\big(T_C(\zb)\big). \label{inclusion} \end{equation}

For a closed convex set $C$ and $(\zb,\zba)\in \Gr N_C$ we have $\Lsp\big(T_C(\zb)\big)\subseteq [\zba]^\perp$ and thus
\begin{align}\label{EqLspCritCone}\Lsp\big(\K_C(\zb,\zba)\big)&= \big(T_C(\zb)\cap[\zba]^\perp\big)\cap\big(-T_C(\zb)\cap[\zba]^\perp\big)=T_C(\zb)\cap(-T_C(\zb)\big)\cap[\zba]^\perp\\
\nonumber&=\Lsp\big(T_C(\zb)\big)
\end{align}

\begin{definition}[Affine Hull] For a closed convex set $C$ we denote by
\[C^+:=\Span (C - C)\]
the unique subspace parallel to the affine hull of $C$.
\end{definition}
\noindent If $K$ is a closed convex cone then we always have
$K^+=K-K.$

{Let $C$ be a closed convex set. Then the tangent cone $T_C(\bar z)$ is a closed convex cone for any $\bar z\in C$.
Since $C\subset \bar z+T_C(\bar z)$ and
$T_C(\bar z)=\limsup_{t\downarrow 0}(C-\bar z)/t\subseteq \Span(C-C)=C^+$ for any $\bar z\in C$, we have
\[C^+=\Span (C-C)\subseteq \Span (T_C(\bar z) -T_C(\bar z))=T_C(\bar z) -T_C(\bar z)\subseteq C^+-C^+=C^+.\]
It follows that for a closed convex set $C$ and  every $\zb\in C$ we have
\begin{equation}\label{affinehull}
C^+=T_C(\zb)-T_C(\zb)=T_C(\zb)^+\end{equation}
and for every $(\zb,\zba)\in \Gr N_C$
\[\K_C(\zb,\zba)^+=T_C(\zb)\cap[\zba]^\perp-T_C(\zb)\cap[\zba]^\perp\subseteq T_C(\zb)^+ \cap [\zba]^\perp.\]}

\if{Since $C\subset \bar z+T_C(\bar z)$ for any $\bar z \in C$,  we have
$$C^+=\Span (C-C)\subseteq \Span (T_C(\bar z) -T_C(\bar z))=T_C(\bar z) -T_C(\bar z).$$
It follows that for a closed convex set $C$ and  every $\zb\in C$ we have
\[C^+=T_C(\zb)-T_C(\zb)=T_C(\zb)^+\]
and for every $(\zb,\zba)\in \Gr N_C$
\[\K_C(\zb,\zba)^+=T_C(\zb)\cap[\zba]^\perp-(T_C(\zb)\cap[\zba]^\perp)\subseteq T_C(\zb)^+ \cap [\zba]^\perp.\]}
\fi

For every closed convex set $C$ and every $\zb\in C$ we have by virtue of Proposition \ref{Prop1} that
\begin{align}\label{EqAnnLspT_C}\Lsp\big(T_C(\zb)\big)^\perp=\big(T_C(\zb)\cap(-T_C(\zb))\big)^\circ={\cl\big(N_C(\zb)-N_C(\zb)\big)}=N_C(\zb)^+,\end{align}

and hence
\begin{align}\label{EqLspT_C}\Lsp\big(T_C(\zb)\big)=\big(N_C(\zb)^+\big)^\perp.\end{align}
Further, by virtue of Proposition \ref{Prop1} we have
\[\Lsp(N_C(\zb))=\big(T_C(\zb)\big)^\circ\cap\big(-T_C(\zb)\big)^\circ=(T_C(\zb)-T_C(\zb))^\circ=(T_C(\zb)^+)^\perp=(C^+)^\perp\]
implying
\begin{equation}\label{cc+}\Lsp(\Gr N_C)=\Lsp(C)\times (C^+)^\perp.\end{equation}

If $C$ is convex polyhedral, then for every $(\zb,\zba)\in\Gr N_C$ we obtain by virtue of \eqref{EqTanNormalConePolyhedral}, (\ref{cc+}),  \eqref{EqLspCritCone} and \eqref{EqLspT_C}
\begin{align}\label{EqLspTanGraphNormalC}
\Lsp\big(T_{\Gr N_C}(\zb,\zba)\big)&=\Lsp\big(\Gr N_{\K_C(\zb,\zba)}\big)=\Lsp\big(\K_C(\zb,\zba)\big)\times \big(\K_C(\zb,\zba)^+\big)^\perp \\
\nonumber &= \Lsp(T_C(\zb))\times \big(\K_C(\zb,\zba)^+\big)^\perp=\big(N_C(\zb)^+\big)^\perp\times \big(\K_C(\zb,\zba)^+\big)^\perp.
\end{align}
Further, for every $(\delta \zb,\delta\zba)\in \Gr N_{\K_C(\zb,\zba)}$, by (\ref{inclusion})  we have
\begin{align}\label{EqLspTanGrNormalCritCone}
  \Lsp\big(T_{\Gr N_{\K_C(\zb,\zba)}}(\delta\zb,\delta\zba)\big)& \supseteq \Lsp\big(\Gr N_{\K_C(\zb,\zba)}\big)=\Lsp\big(\K_C(\zb,\zba)\big)\times\big(\K_C(\zb,\zba)^+\big)^\perp\\
  &=\Lsp\big(T_C(\zb)\big)\times \big(\K_C(\zb,\zba)^+\big)^\perp=\Lsp\big(T_{\Gr N_C}(\zb,\zba)\big),\nonumber
\end{align}
where the equalities follow from (\ref{EqLspTanGraphNormalC}).

\if{
Note that for every closed convex cone $K$ we have $(K^\circ)^-=\{v^\ast \mv\skalp{v^\ast,v}=0 \forall v\in K\}=(K^+)^\perp$, and
\begin{equation}\label{EqSubSpaceN_K}\Gr N_K+\big(K^-\times(K^\circ)^-\big)\subseteq \Gr N_K.\end{equation}
Moreover, for a closed convex set $C$ and $z^\ast\in  N_C(z)$ we have
\[\Big(\big(\K_C(z,z^\ast)\big)^-\Big)^\circ=\Big(T_C(z)\cap\big(-T_C(z)\big) \cap [z^\ast]^\perp\Big)^\circ=N_C(z)+(-N_C(z))+[z^\ast]=\big(N_C(z)\big)^+.\]}\fi

In the following proposition we recall some basic properties of convex polyhedral cones.
\begin{proposition}\label{PropBasicPropCone}
  Consider two finite index sets $I_1$, $I_2$, vectors $a_i\in\R^n$, $i\in I_1\cup I_2$ and let
  \[K:=\left \{v\mv a_i^Tv\begin{cases}=0&i\in I_1,\\\leq 0&i\in I_2\end{cases}\right \}.\]
  Then
  \[\Lsp(K)=\left \{v\mv a_i^Tv=0,\ i\in I_1\cup I_2 \right \},\ K^+\subseteq \left \{v\mv a_i^Tv=0,\ i\in I_1 \right \}\]
  and for every $v\in K$ we have
  \[T_K(v)= \left \{u\mv a_i^Tu\begin{cases}=0&i\in I_1\\\leq 0&i\in I(v)\setminus I_1\end{cases}\right\},\quad N_K(v)=\left \{\sum_{i\in I(v)}\mu_ia_i\mv \mu_i\geq 0, i\in I(v)\setminus I_1 \right \}\]
  where $I(v):=\{i\in I_1\cup I_2\mv a_i^Tv=0\}$. Further, for every $z^*\in N_K(v)$ there is an index set $\I$ with $I_1\subseteq  \I\subseteq  I(v)$ such that
  \[\K_K(v,z^*)=\Big\{u\mv a_i^Tu\begin{cases}=0&i\in\I\\\leq 0&i\in I(v)\setminus \I\end{cases}\Big\}\]
  and vice versa. The faces of $K$  are given by the sets
  \[\F=\left \{u\mv a_i^Tu\begin{cases}=0&i\in\I\\\leq 0&i\in I_2\setminus \I\end{cases}\right \},\mbox{ where } \I \mbox{ satisfies } \ I_1\subseteq \I\subseteq  I_1\cup I_2.\]
  For all $v\in K$, the following face of $K$ defined by
  \[\F_v:=\left \{u\mv a_i^Tu\begin{cases}=0&i\in I(v)\\\leq 0&i\in I_2\setminus I(v)\end{cases}\right \}\]
 {is} the  unique face  satisfying $v\in\ri \F_v$.
  Consequently, for all $v\in K$ and all faces $\F_1,\F_2$ of $K$ such that $v\in\ri \F_2\subseteq  \F_1$ there is some index set $\I$, $I_1\subset\I\subseteq  I(v)$ such that
  \[\F_1-\F_2=\Big\{u\mv a_i^Tu\begin{cases}=0&i\in\I\\\leq 0&i\in I(v)\setminus \I\end{cases}\Big\},\]
  which is the same as saying that there is some $z^*\in N_K(v)$ with $\F_1-\F_2=\K_K(v,z^*)$.
  \end{proposition}

The following lemma will be useful for our analysis:
\begin{lemma}\label{LemAux}Let $\tilde P=(\tilde P_1,\tilde P_2):\R^d\to\R^s\times\R^s$ be continuously differentiable, let $C\subseteq  \R^s$ be a polyhedral convex set and let $\zb\in \R^d$ with
$\tilde P(\zb)\in \tilde D:=\Gr N_C$ be given. Further assume that we are given two subspaces $L_1\supseteq \big(N_C(\tilde P_1(\zb))\big)^+$ and $L_2\supseteq \big(\K_C(\tilde P_1(\zb),\tilde P_2(\zb))\big)^+$ such that
\begin{equation}\label{EqRelaxedMo} \ker\nabla\tilde P(\zb)^T\cap (L_1\times L_2)=\{(0,0)\}.\end{equation}
Then the mapping $z\rightrightarrows \tilde P(z)-\tilde D$ is metrically regular around $(\zb,0)$,
\begin{equation}\label{EqTanPropAux}T_{\{z\mv \tilde P(z)\in\tilde D\}}(\zb)=\{w\mv \nabla\tilde P(\zb)w\in T_{\tilde D}(\tilde P(\zb))\}\end{equation}
and
\begin{eqnarray}\label{EqNormalPropAux}
\lefteqn{\widehat N_{\{z\mv \tilde P(z)\in\tilde D\}}(\zb)= \nabla \tilde P(\zb)^T \widehat N_{\tilde D}(\tilde P(\zb))}\nonumber\\
&&=
 \nabla \tilde P_1(\zb)^T\big(\K_C(\tilde P_1(\zb),\tilde P_2(\zb))\big)^\circ+\nabla \tilde P_2(\zb)^T\K_C(\tilde P_1(\zb),\tilde P_2(\zb)).\end{eqnarray}
\end{lemma}
\begin{proof}
  In order to prove metric regularity of the mapping $\tilde P(\cdot)-\tilde D$ we invoke the Mordukhovich criterion \eqref{EqMoCrit}, which reads in our case as
  \begin{equation}\label{EqAuxMordukhovich}\nabla \tilde P_1(\zb)^Tw^\ast+\nabla \tilde P_2(\zb)^Tw=0,\ (w^\ast,w)\in N_{\Gr N_C}(\tilde P(\zb))\ \Rightarrow\ (w^\ast, w)=0.\end{equation}
  Consider $(w^\ast,w)\in N_{\Gr N_C}(\tilde P(\zb))$. By Proposition \ref{PropNormalConeMapping} there are faces $F_1,F_2$ of the convex polyhedral cone $\tilde K:=\K_C(\tilde P_1(\zb),\tilde P_2(\zb))$ such that  $(w^\ast,w)\in (F_1-F_2)^\circ\times (F_1-F_2)$. Since the lineality space of a convex polyhedral cone is always contained in any of its faces, we have
  $\Lsp(\tilde K)\subseteq F_1-F_2\subset \tilde K^+$, from which  we obtain
  \[(w^\ast,w)\in \Lsp(\tilde K)^\circ\times \tilde K^+=\Lsp(\tilde K)^\perp\times \tilde K^+=\big(N_C(\tilde P_1(\zb))\big)^+\times \big(\K_C(\tilde P_1(\zb),\tilde P_2(\zb))\big)^+\subseteq L_1\times L_2,\]
  where the second equality follows by using \eqref{EqLspCritCone} and \eqref{EqAnnLspT_C}.
  Thus \eqref{EqAuxMordukhovich} follows from \eqref{EqRelaxedMo} and the claimed property of metric regularity is established. Metric regularity in turn implies  MSCQ for the system $\tilde P(z)\in\tilde D$ at $\zb$ and \eqref{EqTanPropAux} follows from \cite[Proposition 1]{HenOut05}. In order to show \eqref{EqNormalPropAux} we will invoke \cite[Theorem 4]{GfrOut16b}. From \eqref{EqLspTanGraphNormalC} we deduce
  \[\Lsp\big(T_{\tilde D}(\tilde P(\zb))\big)= \Big(N_C\big(\tilde P_1(\zb)\big)^+\Big)^\perp\times \Big(\K_C\big(\tilde P_1(\zb),\tilde P_2(\zb)\big)^+\Big)^\perp\supseteq L_1^\perp\times L_2^\perp\]
  and together with \eqref{EqRelaxedMo} we obtain
  \begin{align*}\R^s\times\R^s&=\Big(\ker\nabla\tilde P(\zb)^T\cap(L_1\times L_2)\Big)^\perp=\range \nabla\tilde P(\zb)+ (L_1^\perp\times L_2^\perp)\\
  &\subseteq \range \nabla\tilde P(\zb)+ \Lsp\big(T_{\tilde D}(\tilde P(\zb))\big)\subseteq \R^s\times\R^s,\end{align*}
  where the second equality follows from Proposition \ref{calpolar}.
  Hence the assumption of \cite[Theorem 4]{GfrOut16b} is fulfilled and the first equation in \eqref{EqNormalPropAux} follows, whereas the second equation is a consequence of \eqref{EqNormalConePolyhedral}.
\end{proof}

\section{Optimality conditions for a set-constrained optimization problem}
In this section we consider an optimization problem of the form
\begin{eqnarray}\label{EqGenOptProbl} \min&&f(z)\\
\nonumber\mbox{s.t.} &&   P(z)\in D,
\end{eqnarray}
where $f:\R^d\to\R$ and $P:\R^d\to\R^s$ are continuously differentiable and $D
\subseteq \R^s$ is closed.

Let $\zb$ be a local minimizer and  $\Omega:=\{z\mv P(z)\in D\}$  the feasible region for the problem \eqref{EqGenOptProbl}. Then $ \nabla f(\zb){^T}u \geq 0\ \ \forall u \in T_\Omega(\zb) $ and so the following basic/geometric optimality condition holds:
\begin{equation}\label{EqBasicOpt} 0\in\nabla f(\zb)+ \widehat N_\Omega(\zb).\end{equation}  To express the basic optimality condition in terms of the problem data $P(\cdot)$ and $D$, one needs to estimate  the regular normal cone $\widehat N_\Omega(\zb)$. Given $\zb\in\Omega$ we denote the {\em linearized tangent cone to $\Omega$ at $\zb$} by
\[\TlinO(\zb):=\{u\in\R^d\mv \nabla P(\zb)u\in T_D(P(\zb))\}.\]

It is well known that the inclusions
\begin{gather}\label{EqCQ}T_\Omega(\zb)\subseteq \TlinO(\zb),\ \widehat N_\Omega(\zb)\supseteq \big(\TlinO(\zb)\big)^\circ,\\
\label{EqInclRegNormalCone}\big(\TlinO(\zb)\big)^\circ \supseteq \nabla P(\zb)^T\widehat N_D\big(P(\zb)\big)
\end{gather}
are always valid, cf. \cite[Theorems 6.31, 6.14]{RoWe98}.
Hence, if both inclusions hold with equality, the formula  $\widehat N_\Omega(\zb)=\nabla P(\zb)^T\widehat N_D\big(P(\zb)\big)$ is at our disposal and the optimality condition \eqref{EqBasicOpt} reads as
\[0\in \nabla f(\xb)+\nabla P(\zb)^T\widehat N_D\big(P(\zb)\big),\]
which is also known as {\em strong (S-) stationarity condition}, cf. \cite{FleKanOut07}.
In order to ensure equality in \eqref{EqCQ} one has to impose some constraint qualification.
\begin{definition}\label{DefCQ}
  Let  $P(\zb)\in D$.
  \begin{enumerate}
    \item[(i)] (cf. \cite{FleKanOut07}) We say that the {\em generalized Abadie constraint qualification} (GACQ) holds at $\zb$ if
    \begin{equation*}
      T_\Omega(\zb)=\TlinO(\zb).
    \end{equation*}
  \item[(ii)] (cf. \cite{FleKanOut07}) We say that the {\em generalized Guignard constraint qualification} (GGCQ) holds at $\zb$ if
    \begin{equation}
      \label{EqGGCQ}\widehat N_\Omega(\zb)=\big(\TlinO(\zb)\big)^\circ.
    \end{equation}
    \item[(iii)]  (cf. \cite{GfrMo15a}) We say that the {\em metric subregularity constraint qualification (MSCQ)} holds at $\zb$ for the system $P(z)\in D$ if the set-valued map $M(z):=P(z)-D$ is metrically subregular at $(\zb,0)$.
    \end{enumerate}
\end{definition}
There hold the following implications:
\[\text{MSCQ}\ \Longrightarrow\ \text{GACQ}\ \Longrightarrow\ \text{GGCQ}.\]
Indeed, the first implication follows from \cite[Proposition 1]{HenOut05} whereas the second one is an immediate consequence of the definition of the regular normal cone. GGCQ is the weakest of the three constraint qualifications ensuring $\widehat N_\Omega(\zb){=}
 \big(\TlinO(\zb)\big)^\circ$, but it is very difficult to verify it in general. On the other hand, MSCQ is  stronger than GGCQ but there are effective tools for verifying it.

Now let us {consider  inclusion} \if{examine conditions ensuring equality in}\fi \eqref{EqInclRegNormalCone}. By \cite[Corollary 16.3.2]{Ro70} we have
\[\cl \Big(\nabla P(\zb)^T\widehat N_D\big(P(\zb)\big)\Big)=\big\{w\mv \nabla P(\zb)w\in \cl\co T_D\big(P(\zb)\big)\big\}^\circ\]
showing that we can expect equality in \eqref{EqInclRegNormalCone} only under some restrictive assumption whenever $T_D\big(P(\zb)\big)$ is not convex. Such an assumption is e.g. provided by \cite[Theorem 4]{GfrOut16b}. If it does not hold, but MSCQ holds at $\zb$, then  it is well-known that
$ \widehat{N}_\Omega(\zb) \subseteq N_\Omega(\zb) \subseteq \nabla P(\zb)^TN_D\big(P(\zb)$ and hence the M-stationary condition
\[0\in \nabla f(\zb)+\nabla P(\zb)^TN_D\big(P(\zb)\big)\]
holds at any local optimal solution $\zb$.
For the case where the set $D$ is simple, e.g., $D:=\{(a,b)\mv 0\leq a\perp b \geq 0\}$,  the complementarity cone, the limiting normal cone can be calculated using the variational analysis (cf. \cite{Mor}) and one obtains the classical   M-stationary condition for MPCC.  However, for more complicated set $D$,  e.g., $D:=\Gr \widehat N_\Gamma\times \R^p_-$, usually very strong assumptions are required for using these calculus rules limiting considerably their applicability; see e.g.  Gfrerer and Outrata \cite[Theorem 4]{GfrOut16a} .

Recently an alternative approach is taken by Gfrerer in \cite{Gfr18}.
Under GGCQ, by (\ref{EqGGCQ}) for every regular normal $z^\ast\in \widehat N_\Omega(\zb)$ the point $u=0$ is a global minimizer for the problem
\[\min_u -{z^*}^Tu \quad\text{subject to}\quad \nabla P(\zb)u\in T_D\big(P(\zb)\big).\]
Provided that the mapping $u\rightrightarrows \nabla P(\zb)u - T_D\big(P(\zb)\big)$ is metrically subregular at $(0,0)$ we can apply the M-stationarity conditions to this linearized problem which read as
\[z^\ast\in \nabla P(\zb)^TN_{T_D\big(P(\zb)\big)}(0).\]
Thus we obtain the inclusion
$\widehat N_\Omega(\zb)\subseteq \nabla P(\zb)^TN_{T_D\big(P(\zb)\big)}(0).$ This results in
a necessary optimality condition \begin{equation}\label{shaperc}
0\in \nabla f(\zb)+ \nabla P(\zb)^TN_{T_D\big(P(\zb)\big)}(0),\end{equation}
which is  sharper than the M-stationarity condition since
 $N_{T_D\big(P(\zb)\big)}(0)\subseteq N_D\big(P(\zb)\big)$, cf. \cite[Proposition 6.27(a)]{RoWe98}.
Although (\ref{shaperc}) is a sharper condition than the M-stationary condition, it still involves the limiting normal cone and so may be hard to calculate.
In  \cite[Propositions 1,2]{Gfr18}, Gfrerer derived the following linearized M-necessary optimality condition  which can be considered as a refinement of the necessary optimality condition   (\ref{shaperc}). The condition is easier to calculate since it involves only the regular normal cone. In fact by virtue of Lemma \ref{Normalcones}, in the case where $T_D\big(P(\zb))$ is locally polyhedral at $0$,  condition (\ref{EqKKT}) coincides with  condition (\ref{shaperc}).
\begin{theorem}\label{ThLinM_StatBasic}
  Let $\zb$ be a local optimal solution for problem  \eqref{EqGenOptProbl}. Assume that  GGCQ holds at  $\zb$ and the mapping $u\rightrightarrows \nabla P(\zb)u- T_D(P(\zb))$ is metrically subregular at $(0,0)$. Then  one of the following two conditions is fulfilled:
  \begin{enumerate}
    \item[(i)]  There is $\omega\in T_D(P(\zb))$   such that
  \begin{equation}
    \label{EqKKT}0\in \nabla f(\zb)+ \nabla P(\zb)^T \widehat N_{T_D(P(\zb))}(\omega).
  \end{equation}
    \item[(ii)]  There is $\ub\in \TlinO(\bar z)$ such that
    \begin{eqnarray}
\label{EqNotZero}&&\nabla P(\zb)\ub\not\in \Lsp(T_D(P(\zb))),\\
\label{EqKKTCritDir}      &&\nabla f(\zb) ^T \ub =0,\\
\label{EqKKTNormal}      &&0\in  \nabla f(\zb) +\widehat N_{\TlinO(\zb)}(\ub)
    \end{eqnarray}
    and $T_D(P(\zb))$ is  not locally polyhedral near $\nabla P(\zb)\ub$.
  \end{enumerate}
  If in addition $T_D(P(\zb))$ is the graph of a set-valued mapping $M=M_c+M_p$, where $M_c,M_p:\R^r\rightrightarrows \R^{s-r}$ are set-valued mappings whose graphs are closed cones, $M_p$ is polyhedral and there is some real $C$ such that
  \begin{equation}\label{EqBoundM_C} \norm{t}\leq C \norm{v}\ \forall (v,t)\in\Gr M_c\end{equation}
  then there is some $\bar v\not=0$ such that
    \begin{equation}
    \label{EqKKT_W_NotZero} \nabla P(\zb)\bar u\in\{\bar v\}\times M(\bar v).
  \end{equation}
\end{theorem}
\begin{remark}\label{RemSubReg}
(i)   Note that the assumptions of Theorem \ref{ThLinM_StatBasic} are fulfilled if MSCQ holds at $\zb$. Indeed, MSCQ implies GGCQ and metric subregularity of $u\rightrightarrows \nabla P(\zb)u- T_D(P(\zb))$ at $(0,0)$ follows from \cite[Lemma 4]{Gfr18}.

(ii) Note that when  $T_D(P(\zb))$ is locally polyhedral near $\nabla P(\zb)\ub$, then  condition (i) holds. Otherwise, if $T_D(P(\zb))$ is not locally polyhedral near $\nabla P(\zb)\ub$, then  condition (ii) {can hold. In this case, \eqref{EqKKTCritDir} implies together with GGCQ and the basic optimality condition \eqref{EqBasicOpt} that $\bar u$ is a global solution of the problem
\[\min \nabla f(\zb)^Tu \quad \mbox{ subject to } \nabla P(\zb)u\in T_D(P(\zb)).\]}
 \if{holds. In particular (\ref{EqKKTNormal}) holds. By \cite[Theorems 6.12, 6.11]{RoWe98}, there exists a continuously differentiable function $\widetilde f : \R^{d} \rightarrow \R$ with  $\nabla \widetilde{f}(\ub)=  \nabla f(\zb)$ such that $\ub$ is a global minimizer of the problem
 $$ \min \widetilde  f(u) \quad \mbox{ subject to } \nabla P(\zb)u\in T_D(P(\zb)). $$}\fi
 Now since the graph of the mapping $u\rightrightarrows \nabla P(\zb)u-T_D(P(\zb))$ is a closed cone,  by virtue of  \cite[Lemma 3]{Gfr18}, the metric subregularity of $u\rightrightarrows \nabla P(\zb)u-T_D(P(\zb))$ at $(0,0)$ implies the metric subregularity of the same mapping at $(\bar u,0)$. Hence we can apply Theorem \ref{ThLinM_StatBasic} once more to the above problem. If $T_{T_D(P(\zb))}(\nabla P(\zb)\bar u)$ is   
 polyhedral,
 then Theorem \ref{ThLinM_StatBasic}(i) applies and we obtain
{the} 
 existence of $\omega\in T_{T_D(P(\zb))}(\nabla P(\zb)\bar u)$   such that
  \begin{equation}
    \label{EqKKTnew}0\in \nabla f(\zb)+ \nabla P(\zb)^T \widehat N_{T_{T_D(P(\zb))}(\nabla P(\zb)\bar u)}(\omega).
  \end{equation}
  In this case (\ref{EqKKTnew}) would be the necessary optimality condition which is sharper than condition (\ref{EqKKT}). In this paper we aim at finding some sufficient conditions, i.e., the 2-nondegeneracy condition introduced in Subsection \ref{SubSecTwoNondegen} below,  under which the tangent cone $T_{T_D(P(\zb))}(\nabla P(\zb)\bar u)$ is polyhedral and hence the above  optimality condition holds for (MPEC).
  However, in general  $T_{T_D(P(\zb))}(\nabla P(\zb)\bar u)$ is  not polyhedral,
   and the process {could} continue. The interested reader is referred to \cite{Gfr18} for  the discussion for what might have happened after applying Theorem \ref{ThLinM_StatBasic} repeatedly.
\end{remark}
\section{Constraint qualifications for the new optimality conditions}

 Note that (MPEC) can be written in the form \eqref{EqGenOptProbl} via
\begin{eqnarray}\label{EqMPECAlt}\min_{x,y}&&F(x,y)\\
\nonumber\mbox{subject to}&&P(x,y):=\left(\begin{array}{c}(y,-\phi(x,y))\\G(x,y)\end{array}\right)\in D:=\Gr \widehat N_\Gamma\times \R^p_-,
\end{eqnarray}
where $\Gamma:= \{y\mv g(y)\leq 0\}$.
To apply Theorem \ref{ThLinM_StatBasic},  we will need the following assumptions at a local solution $(\xb,\yb)$ to problem  \eqref{EqMPECAlt}.
\begin{assumption}\label{Ass1}
\begin{enumerate}\item[(i)] MSCQ holds for the lower level constraint $g(y)\in\R^q_-$ at $\yb$.
\item[(ii)] GGCQ holds at $(\xb,\yb)$ and the mapping
\begin{eqnarray*}(u,v)&\rightrightarrows& \nabla P(\xb,\yb)(u,v)-T_D((\yb,\yba),G(\xb,\yb))\\&=&\left(\begin{array}{c}(v,-\nabla \phi(\xb,\yb))(u,v)\\\nabla G(\xb,\yb)(u,v)\end{array}\right)-T_{\Gr \widehat N_\Gamma\times \R^p_-}((\yb,\yba),G(\xb,\yb))\end{eqnarray*}
is metrically subregular at $((0,0),0)$, where $\yba:=-\phi(\xb,\yb)$.
\end{enumerate}
\end{assumption}
It is well-known that  if  either $g$ is affine or the Mangasarian-Fromovitz constraint qualification (MFCQ) holds at
$\zb$ then  Assumption \ref{Ass1}(i) holds.
By Remark \ref{RemSubReg},  Assumption \ref{Ass1}(ii)  is fulfilled if MSCQ holds for the system $P(x,y)\in D$ at $(\xb,\yb)$. A point-based sufficient condition for the validity of MSCQ for this system is given by \cite[Theorem 5]{GfrYe16a}.
 We now describe this condition. {When MSCQ holds at $\bar y$ for the system $g(y)\in\R^q_-$, we have $T_\Gamma(\bar y)=\{v\mv \nabla g(\yb)v\in T_{\R^q_-}(g(\yb))\}$ and thus the critical cone $\K_\Gamma(\yb,\yba)$ amounts to
\begin{equation}\label{criticalcone-lowerlevel}
\KbG:=\K_\Gamma(\yb, \yba)=\{v\mv \nabla g(\yb)v\in T_{\R^q_-}(g(\yb))\}\cap[\yba]^\perp,\end{equation}
which is   convex polyhedral.}
  Further we define the {\em multiplier set} for $\bar y$ as the polyhedral convex {set} 
 defined by
\begin{equation}\label{multiplierset} \Lb:=\Lambda(\bar y,\bar y^*):=\{\lambda\in N_{\R^q_-}(g(\yb))\mv \nabla g(\yb)^T\lambda=\yba\}.\end{equation}
{For a multiplier $\lambda$, the corresponding collection of {\em strict complementarity indexes} is denoted by}
\begin{eqnarray*}
I^+(\lambda):=\big\{i\in \{1,\ldots,q\}\big|\;\lambda_i>0\big\}\;\mbox{ for }\;\lambda=(\lambda_1,\ldots,\lambda_q)\in\R^q_+.
\end{eqnarray*}
Denote by $\E(\yb,\bar y^*)$ the collection of all the extreme points of the closed and convex set of multipliers $\Lambda(\bar y,\bar y^*)$ and recall that $\lambda \in \Lambda(\bar y,\bar y^*)$ belongs to $\E(\yb,\bar y^*)$ if and only the family of gradients $ \{ \nabla g_i(\bar y)| i\in I^+(\lambda)\}$ is linearly independent.
Moreover  for every $v\in {\cal K}_\Gamma(\bar y, \bar y^*)$, we define  the {\em directional multiplier set} as
\begin{equation*}\Lbv:=\Lambda(\bar y,\bar y^*;v):=\argmax \left\{v^T\nabla^2(\lambda^Tg)(\yb)v\mv \lambda\in \Lb \right \}, \end{equation*}
which is also a polyhedral convex {set.} 
By \cite[Proposition 4.3(iii)]{GfrMo15a} we have $\Lbv\not=\emptyset$ $\forall v\in\KbG$ under Assumption \ref{Ass1}.
Note that since $\bar{\Lambda}$ is a closed convex set and the objective of the above problem is linear, by the optimality condition,
\begin{equation}\label{EqChar_lambda}
\lambda\in\Lbv\Longleftrightarrow v^T\nabla^2g(\yb) v\in N_{\Lb}(\lambda).
\end{equation}
\begin{theorem}[{cf. \cite[Theorem 4]{GfrYe16a}}]\label{ThTanConeGrNormalCone}Let $\bar y\in \Gamma:= \{y\mv g(y)\leq 0\} $ and $\bar y^*=-\phi(\bar x,\bar y)$.  Assume that MSCQ holds  at $\yb$ for
the system $g(y)\in\R^q_-$. Then
 the tangent cone to the graph of $\widehat N_\Gamma$ at $(\yb,\yba)$ can be calculated by
\begin{eqnarray}\label{EqTanConeGrNormalCone}
\lefteqn{T_{\Gr \widehat N_\Gamma}(\yb,\yba)}\\
\nonumber
&=&\big\{(v,v^\ast)\in\R^{2m}\mv\exists\,\lambda\in\Lbv\;\mbox{ with }\;
v^\ast\in\nabla^2(\lambda^Tg)(\yb)v+N_{\KbG}(v)\big\}\\
\nonumber&=&\big\{(v,v^\ast)\in\R^{2m}\mv\exists\,\lambda\in\Lbv\cap \kappa\norm{\yba} \B_{\R^q}\;\mbox{ with }\;
v^\ast\in\nabla^2(\lambda^Tg)(\yb)v+N_{\KbG}(v)\big\},
\end{eqnarray}
where $\kappa>0$ is certain constant.
Or equivalently \begin{eqnarray}\label{EqTanConeGrNormalConeAlt}
\lefteqn{T_{\Gr \widehat N_\Gamma}(\yb,\yba)}\\
\nonumber
&=&\big\{(v,\nabla^2(\lambda^Tg)(y)v+z^\ast)\in\R^{2m}\mv  (z^\ast,v^T\nabla^2g(\yb)v)\in N_{\KbG\times\Lb}(v,\lambda)\big\}\\
\nonumber&=&\big\{(v,\nabla^2(\lambda^Tg)(y)v+z^\ast)\in\R^{2m}\mv  (z^\ast,v^T\nabla^2g(\yb)v)\in N_{\KbG\times\Lb}(v,\lambda),\ \norm{\lambda}\leq \kappa\norm{\yba}\big\}.
\end{eqnarray}
\end{theorem}
\noindent Note that the equivalence of (\ref{EqTanConeGrNormalCone}) and (\ref{EqTanConeGrNormalConeAlt}) is s due to   \eqref{EqChar_lambda}.

We now in a position to review a  sufficent condition for  Assumption \ref{Ass1} to hold.
\begin{theorem}[{ \cite[Theorem 5]{GfrYe16a}}]\label{ThSuffCondMS_FOGE}
  Let $(\xb,\yb)$ be  a feasible solution of  the system $P(x,y)\in D$. Assume that  MSCQ holds both  for the lower level problem constraints $g(y)\leq 0$ at $\yb$ and for the upper  level constraints $G(x,y)\leq0$ at
  $(\xb,\yb)$. Further assume that
  \begin{equation*}
  \nabla_x G(\xb,\yb)^T\eta=0,\ \eta\in N_{\R^p_-}(G(\xb,\yb))\quad \Longrightarrow\quad \nabla_y G(\xb,\yb)^T\eta=0,
  \end{equation*}
  and assume that
  there do not exist $(u,v)\not=0$, $\lambda\in\Lambda(\yb,-\phi(\xb,\yb);v)\cap \E(\yb,-\phi(\xb,\yb))$, $\eta\in\R^p_+$ and $w\not=0$ satisfying
  \begin{eqnarray*}
    &&\nabla G(\xb,\yb)
   (u,v)\in T_{\R^p_-}(G(\xb,\yb)),\; \\
    &&(v,-\nabla_{x}\phi(\xb,\yb)u-\nabla_{y}\phi(\xb,\yb)v)\in T_{\Gr \widehat N_\Gamma}(\yb,-\phi(\xb,\yb)),\\
    &&-\nabla_{x}\phi(\xb,\yb)^Tw+\nabla_{x} G(\xb,\yb)^T\eta=0,\;\eta\in N_{\R^p_-}(G(\xb,\yb)),\; \eta^T\nabla G(\xb,\yb)(u,v)=0,\\
    &&\nabla g_i(\yb)w=0, i\in I^+(\lambda),\; w^T\left (\nabla_{y}\phi(\xb,\yb)+\nabla^2(\lambda^Tg(\yb)\right )w-\eta^T\nabla_y G(\xb,\yb)w\leq 0.\qquad
  \end{eqnarray*}
Then MSCQ for the system $P(x,y)\in D$ holds at $(\xb,\yb)$.
\end{theorem}

\section{Computing regular normals to tangent directions and tangents of tangents}

In this section, we apply Theorem \ref{ThLinM_StatBasic} to obtain a necessary optimality condition for program \eqref{EqMPECAlt} {which is equivalent to the MPEC \eqref{EqMPEC}.
 In order to  apply Theorem \ref{ThLinM_StatBasic}, for any $(\vb,\vba,a):=w \in T_D(P(\xb,\yb))$, we need to compute
$\widehat N_{T_{D}(P(\xb,\yb))}(w)$.  Using \cite[Proposition 1]{GfrYe16a} together with \cite[Proposition 6.41]{RoWe98}, we obtain
$$\widehat N_{T_{D}(P(\xb,\yb))}(w) =\widehat N_{T_{\Gr \widehat N_\Gamma}(\yb,\yba)}(\vb,\vba)\times \widehat N_{T_{\R^p_-}(G(\xb,\yb))}(a).$$
Hence the aim of this section is to   compute  the regular tangent cone to the tangent directions $\widehat N_{T_{\Gr \widehat N_\Gamma}(\yb,\yba)}(\vb,\vba)$. Similarly by virtue of (\ref{EqKKTnew}) we also need to compute  the regular tangent cone to the tangents of tangents  $\widehat N_{T_{T_{\Gr \widehat N_\Gamma}(\yb,\yba)}(\vb,\vba)}(\delta\vb,\delta\vba)$.

As discussed in the introduction, under a certain  constraint qualification such as  CRCQ, formulas for $\widehat N_{T_{D}(P(\xb,\yb))}(w)$ and the resulting optimality condition for \eqref{EqMPEC}  are derived in \cite[Proposition 3, Theorem 5]{Gfr18}. In this paper we use a different approach.
Given $(\vb,\vba )\in T_{\Gr \widehat N_\Gamma}(\yb,\yba)$, by the formula for the tangent cone  \eqref{EqTanConeGrNormalCone},  there is some $\lambda \in \Lb(\vb)$ and  $z^*\in N_{\KbG}(\vb)$
such that
$\vba \in \nabla^2(\lambda^Tg)(\yb)\vb+z^*.$ Suppose that such representation is unique, i.e., there is a unique $\bar \lambda\in \Lb(\vb)$ and $\bar z^*\in N_{\KbG}(\vb)$ such that $\vba =\nabla^2(\bar \lambda^Tg)(\yb)\vb+\bar z^*.$ Then
$$(\vb,\vba)\in T_{\Gr \widehat N_\Gamma}(\yb,\yba) \Longleftrightarrow \vba =\nabla^2(\bar \lambda^Tg)(\yb)\vb+\bar z^*.$$
The uniqueness allows the efficient calculation of the regular normal cone to tangent directions. To guarantee this uniqueness we perform our analysis under the assumption of 2-nondegeneracy on $g$ as introduced in the next subsection.

\subsection{2-nondegeneracy\label{SubSecTwoNondegen}}


\begin{definition}\label{two-nond} Let $\vb\in \KbG$. We say that $g$ is {\em 2-nondegenerate in direction $\vb$ at $(\yb,\yba)$} if
\[\nabla^2(\mu^Tg)(\yb)\vb \in (N_\KbG(\vb))^+,\ \mu\in \big(\Lb(\vb)\big)^+ \ \Longrightarrow\ \mu=0.\]
\end{definition}
\noindent In the case where the {directional} multiplier set $\Lb(\vb)$ is a singleton, $\big(\Lb(\vb)\big)^+=\{0\}$ and hence $g$ is  2-nondegenerate in this direction $\vb$.
In particular, if $\Lb$ is a singleton then $g$ is 2-nondegenerate in any direction $\vb$.
We now provide a formulation of 2-nondegeneracy in terms of index sets. To this end let us define
\begin{gather*}
  \bar I:=\{i\in\{1,\ldots,q\}\mv g_i(\yb)=0\},\quad \bar I(v):=\{i\in \bar I\mv \nabla g_i(\yb)^T v=0\}, v\in\KbG,\\
\bar J^+(\lambda):=\{i\in \bar I\mv \lambda_i>0\},\lambda\in\Lb,\quad \bar J^+(\Xi):=\bigcup_{\lambda\in\Xi}\bar J^+(\lambda) {\mbox{ for any }}  \Xi\subseteq \Lb.
\end{gather*}
By the definition of the critical cone in (\ref{criticalcone-lowerlevel}), we have
$$\KbG=\Big\{v\mv \nabla g_i(\yb)^T v\leq 0,\ i\in\bar I\Big\}\cap[\yba]^\perp .$$
Since by the definition of the multiplier set (\ref{multiplierset}),
$$\lambda\in\Lb\Longleftrightarrow \lambda\in N_{\R^q_-}(g(\yb)),  \quad \yba=\nabla g(\yb)^T\lambda$$
we have
$$v\in [\yba]^\perp \Longleftrightarrow 0=\yba^Tv=\langle\nabla g(\yb)^T\lambda, v\rangle= \lambda^T  \nabla g(\yb)v. $$
Hence it is obvious that for every $\lambda\in\Lb$ we have
\[\KbG=\Big\{v\mv \nabla g_i(\yb)^T v\begin{cases}=0&i\in \bar J^+(\lambda)\\
\leq 0&i\in\bar I\setminus \bar J^+(\lambda)\end{cases}\Big\}\]
yielding
\begin{equation*}\KbG=\Big\{v\mv \nabla g_i(\yb)^T v\begin{cases}=0&i\in \bar J^+(\Lb)\\
\leq 0&i\in\bar I\setminus \bar J^+(\Lb)\end{cases}\Big\}=\Big\{v\mv \nabla g_i(\yb)^T v\begin{cases}=0&i\in \bar J^+(\Lb(\vb))\\
\leq 0&i\in\bar I\setminus \bar J^+(\Lb(\vb))\end{cases}\Big\}.\end{equation*}
Thus
\[N_\KbG(\vb)=\{\sum_{i\in \bar I(\vb)}\eta_i\nabla g_i(\yb)\mv \eta_i\geq 0,i\in\bar I(\vb)\setminus \bar J^+(\Lb(\vb))\}\]
 Now choose $\hat J$ with $\bar J^+(\Lb(\vb))\subseteq \hat J\subseteq \bar I(\vb)$ large enough such that for every $j\in \bar I(\vb)\setminus \hat J$ the gradient $\nabla g_j(\yb)$ linearly depend on $\nabla g_i(\yb)$, $i\in\hat J$.
 It follows that
\[\big(N_\KbG(\vb)\big)^+=\{\sum_{i\in \bar I(\vb)}\eta_i\nabla g_i(\yb)\mv \eta_i\in\R,i\in\bar I(\vb)\}=\{\sum_{i\in \hat J}\eta_i\nabla g_i(\yb)\mv \eta_i\in\R,i\in\hat J\}.\]
Next we claim that
\begin{gather*}
\big(\Lb(\vb)\big)^+=L:=\{\mu\in\R^q\mv \nabla g(\yb)^T\mu=0,\ \vb^T\nabla ^2(\mu^Tg)(\yb)\vb=0, \mu_i=0,\ i\not\in \bar J^+(\Lb(\vb))\}.
 \end{gather*}
Indeed, for every pair $\lambda^1,\lambda^2\in \Lb(\vb)$, we have $\yba=  \nabla g(\yb)^T\lambda^1= \nabla g(\yb)^T\lambda^2$ and $$\vb ^T\nabla^2(\lambda_1^T g)(\yb) \vb=\vb ^T\nabla^2(\lambda_2^T g)(\yb) \vb, $$ which implies that
\[\nabla g(\yb)^T(\lambda^1-\lambda^2)=0,\ \vb^T\nabla ^2((\lambda^1-\lambda^2)^Tg)(\yb)\vb=0, \lambda^1_i-\lambda^2_i=0,\ i\not\in \bar J^+(\Lb(\vb))\]
showing $\big(\Lb(\vb)\big)^+\subseteq L$. To show the reverse inclusion, take any $\mu\in L$ and any $\lb\in \Lb(\vb)$.
Then $\lb+\alpha\mu\geq0$ for all $\alpha>0$ sufficiently small. It is easy to see  that $\lb+\alpha\mu\in\Lb(\bar v)$ implying $L\subseteq \big(\Lb(\vb)\big)^+$. Thus our claim holds true and
we obtain that $g$ is {\em 2-nondegenerate in direction $\vb$ at $(\yb,\yba)$} if and only if
\begin{equation}\label{Eq2NondegenIndex}\sum_{i\in\hat J}\eta_i\nabla g_i(\yb)+\sum_{i\in\bar J^+(\Lb(\vb))}\mu_i\nabla ^2g_i(\yb)\vb=0, \ \sum_{i\in\bar J^+(\Lb(\vb))}\mu_i\nabla g_i(\yb)=0\ \Rightarrow \mu_i=0,\ i\in\bar J^+(\Lb(\vb)).
\end{equation}
We now want to compare 2-nondegeneracy with the  notion of {\em 2-regularity} which was initiated (and named) by Tret'yakov \cite{Tr84} in the case of zero Jacobian and then was strongly developed by Avakov \cite{Av85}. A twice continuously differentiable mapping $h:\R^m\to \R^l$ is called 2-regular at a point $\yb\in\R^m$ in direction $v\in\R^m$ if for all $\alpha\in\R^l$ the system
\[\nabla h(\yb)u+v^T\nabla^2h(\yb)w=\alpha,\ \nabla h(\yb)w=0\]
has a solution $(u,w)$.
We claim that 2-regularity of $(g_i)_{i\in \hat J}$ implies 2-nondegeneracy of $g$ in direction $\vb$. Indeed, by the Farkas lemma 2-regularity of $(g_i)_{i\in\hat J}$ in direction $\vb$ is equivalent to the statement
\begin{equation*}\sum_{i\in\hat J}(\eta_i\nabla g_i(\yb)+\mu_i\nabla ^2g_i(\yb)\vb)=0, \ \sum_{i\in\hat J}\mu_i\nabla g_i(\yb)=0\ \Rightarrow \mu_i=0,\ i\in\hat J.\end{equation*}
and it is easy to see that this condition implies \eqref{Eq2NondegenIndex}.

The following lemma states some important consequences of 2-nondgeneracy.
\begin{lemma}\label{LemUnique_lambda}Assume that $g$ is 2-nondegenerate in the critical direction $\vb\in \KbG$ at $(\yb,\yba)$ and define the subspace
\begin{equation*}
  {\cal H}(\vb):=\{\nabla^2(\mu^Tg)(\yb)\vb\mv \mu\in \big(\Lb(\vb)\big)^+\}+ (N_\KbG(\vb))^+.
\end{equation*}
Then the linear mapping $\A_{\vb}:\big(\Lb(\vb)\big)^+\times(N_\KbG(\vb))^+\to {\cal H}(\vb)$ given by
\[\A_{\vb}(\mu,z^\ast):=\nabla^2(\mu^Tg)(\yb)\vb+z^\ast\]
is a bijection. In particular, for every $\vba$ with $(\vb,\vba)\in T_{\Gr \widehat N_\Gamma}(\yb,\yba)$ there are unique elements ${\bar{\lambda}} \in \Lb(\vb)$ and $\zba\in N_{\KbG}(\vb)$
  \begin{equation}\label{EqReprVast}\vba=\nabla^2(\lb^Tg)(\yb)\vb+\zba.\end{equation}
\end{lemma}
\begin{proof}
By the definition, the mapping $\A_{\vb}$ is surjective and therefore we only have to show injectivity. Consider elements $(\mu,\zba)\in \big(\Lb(\vb)\big)^+\times(N_\KbG(\vb))^+$ satisfying $\A_{\vb}(\mu,\zba)=0$. Then  $\nabla^2(\mu^Tg)(\yb)\vb =-\zba\in (N_\KbG(\vb))^+$ and by the assumed 2-nondegeneracy of $g$ in direction $\vb$ we obtain $\mu=0$ and consequently $\zba=0$. Thus $\A_{\vb}$ is injective.

In order to show the second statement consider  $\vba$ with $(\vb,\vba)\in T_{\Gr \widehat N_\Gamma}(\yb,\yba)$. The existence of $(\lb,\zba)\in\Lb(\vb)\times N_{\KbG}(\vb)$ fulfilling \eqref{EqReprVast} follows from Theorem \ref{ThTanConeGrNormalCone}. In order to prove  uniqueness of the representation \eqref{EqReprVast},
consider $(\lambda_1,z_1^\ast),(\lambda_2,z_2^\ast)\in \Lb(\vb)\times N_{\KbG}(\vb)$ such that \[\vba=\nabla^2({\lambda_j}^Tg)(\yb)\vb+z_j^\ast,\ j=1,2\]
implying
\[\A_{\vb}(\lambda_2-\lambda_1, z_2^\ast-z_1^\ast)=\nabla^2 ((\lambda_2-\lambda_1)^Tg)(\yb)\vb+z_2^\ast-z_1^\ast=0.\]
Then $\lambda_2-\lambda_1\in\big(\Lb(\vb)\big)^+$ and $z_2^\ast-z_1^\ast\in (N_{\KbG}(\vb))^+$ and by the injectivity of $\A_{\vb}$ we obtain $\lambda_2=\lambda_1$ and $z_2^\ast=z_1^\ast$.
\end{proof}

\subsection{Regular normals to tangent directions}

Throughout this subsection let $(\vb,\vba)\in T_{\Gr \widehat N_\Gamma}(\yb,\yba)$ be given. The main purpose of this section is to compute  the regular normal cone of the tangent directions $\widehat N_{T_{\Gr \widehat N_\Gamma}(\yb,\yba)}(\vb,\vba)$.

\begin{proposition}\label{PropTanConeTanConeGrNormalCone}Assume that $g$ is 2-nondegenerate in the critical direction $\vb\in \KbG$ at $(\yb,\yba)$. Then for every $\vba$ with $(\vb,\vba)\in T_{\Gr \widehat N_\Gamma}(\yb,\yba)$ we have
\begin{equation}\label{EqTanConeTanConeGrNormalCone}
  T_{T_{\Gr \widehat N_\Gamma}(\yb,\yba)}(\vb,\vba)=\left \{(u, u^*) \mv \exists \mu, \zeta^* \mbox{ s.t. } \begin{array}{l}
  u^*=\nabla^2(\lb^Tg)(\yb)u+\nabla^2(\mu^Tg)(\yb)\vb+\zeta^\ast,\\
 (u,\mu,\zeta^\ast, 2\vb^T\nabla^2g(\yb)u)\in \Gr N_{\tilde K(\vb,\vba)}
 \end{array} \right \},
\end{equation}
where $(\lb,\zba)\in \Lb(\vb)\times N_{\KbG}(\vb)$ is the unique element fulfilling $\vba=\nabla^2(\lb^Tg)(\yb)\vb+\zba$ 
 and
 \begin{equation*}
 \tilde K(\vb,\vba):=\K_{\KbG\times\Lb}(\vb,\lb,\zba,\vb^T\nabla^2g(\yb)\vb).\end{equation*}
Further,
\begin{eqnarray}
 \lefteqn{ \widehat N_{T_{\Gr \widehat N_\Gamma}(\yb,\yba)}(\vb,\vba)}\nonumber\\
 &=&\Big\{(w^\ast,w)\mv \begin{array}{l}
  \exists \eta \mbox{ s.t. } \big(w^\ast+\nabla^2(\lb^Tg)(\yb)w- 2\nabla^2(\eta^Tg)(\yb)\vb, \vb^T\nabla^2 g(\yb)w,w,\eta\big)\\
  \qquad \quad \qquad \quad \in \big(\tilde K(\vb,\vba)\big)^\circ\times \tilde K(\vb,\vba) \end{array}\Big\}.  \label{EqRegNormalConeTanConeGrNormalCone}
\end{eqnarray}

\end{proposition}
\begin{proof}
  Let $\vba$ with $(\vb,\vba)\in T_{\Gr \widehat N_\Gamma}(\yb,\yba)$  be fixed and let ${\cal R}$ denote the set on the right hand side of equation \eqref{EqTanConeTanConeGrNormalCone}.

  {\bf Step 1}. In this step we will show that $T_{T_{\Gr \widehat N_\Gamma}(\yb,\yba)}(\vb,\vba)\subseteq{\cal R}$. Let $(u,u^\ast)\in T_{T_{\Gr \widehat N_\Gamma}(\yb,\yba)}(\vb,\vba)$. Then by definition of the tangent cone, there exists  sequences $t_k\downarrow 0$, $(u_k,u_k^\ast)\to(u,u^\ast)$ with $(\vb+t_ku_k,\vba+t_ku_k^\ast)\in T_{\Gr \widehat N_\Gamma}(\yb,\yba)$. By \eqref{EqTanConeGrNormalCone} there are sequences $\lambda_k\in \Lb(\vb+t_ku_k)\cap\kappa\norm{\yba}\B_{\R^q}$ and $z_k^\ast\in N_\KbG(\vb+t_ku_k)$ such that
  \[\vba+t_ku_k^\ast=\nabla^2(\lambda_k^Tg)(\yb)(\vb+t_ku_k)+z_k^\ast.\]
Moreover, since by Lemma \ref{LemUnique_lambda} there are unique elements ${\bar{\lambda}} \in \Lb(\vb)$ and $\zba\in N_{\KbG}(\vb)$
satisfying $\vba=\nabla^2(\lb^Tg)(\yb)\vb+\zba$, it follows that
\begin{equation}\label{secondorder} \nabla^2((\lambda_k-\lb)^Tg)(\yb)\vb+z_k^\ast-\zba=t_k\big(u_k^\ast-\nabla^2(\lambda_k^Tg)(\yb)u_k\big).\end{equation}
For all $k$ sufficiently large we have $N_\KbG(\vb+t_ku_k)\subseteq N_\KbG(\vb)$ by \eqref{EqIncLTanNormalConePoly} and $\Lb(\vb+t_ku_k)\subseteq\Lb(\vb)$ by \cite[Lemma 3]{GfrOut16b}. Hence we have
$\lambda_k-\lb\in\big(\Lb(\vb)\big)^+$ and $z_k^\ast-\zba$ in $(N_\KbG(\vb))^+$. Thus from (\ref{secondorder}), we have $$\A_{\vb}(\lambda_k-\lb,z_k^\ast-\zba):= \nabla^2((\lambda_k-\lb)^Tg)(\yb)\vb+z_k^\ast-\zba=t_k\big(u_k^\ast-\nabla^2(\lambda_k^Tg)(\yb)u_k\big).$$
   By the boundedness of $\lambda_k$ we conclude $t_k\big(u_k^\ast-\nabla^2(\lambda_k^Tg)(\yb)u_k\big)\to 0$. Hence, by Lemma \ref{LemUnique_lambda}  we have  $(\lambda_k-\lb, z_k^\ast-\zba)\to(0,0)$ and
\[(\mu,\zeta^\ast):=\lim_{k\to\infty}(\mu_k,\zeta_k^\ast)=\lim_{k\to\infty}\A_{\vb}^{-1}(u_k^\ast-\nabla^2(\lambda_k^Tg)(\yb)u_k)=\A_{\vb}^{-1}(u^\ast-\nabla^2(\lb^Tg)(\yb)u),\]
where $\mu_k:=\frac{\lambda_k-\lb}{t_k}$ and $\zeta_k^\ast:= \frac{z_k^\ast-\zba}{t_k}$.
Thus \begin{equation}\label{ustar*}
u^\ast=\nabla^2(\lb^Tg)(\yb)u+\A_{\vb}(\mu,\zeta^\ast)=\nabla^2(\lb^Tg)(\yb)u+\nabla^2(\mu^Tg)(\yb)\vb+\zeta^\ast. \end{equation}  Since
$z_k^\ast\in N_\KbG(\vb+t_ku_k)$ and $\lambda_k\in \Lb(\vb+t_ku_k)$ which is equivalent to saying that $(\vb+t_ku_k)^T\nabla^2g(\yb)(\vb+t_ku_k))\in N_{\bar \Lambda}(\lambda_k)$  by virtue of (\ref{EqChar_lambda}),  we have   $$(\vb+t_ku_k,\lambda_k,z_k^\ast, (\vb+t_ku_k)^T\nabla^2g(\yb)(\vb+t_ku_k))\in \Gr N_{\KbG\times \Lb}.$$
It follows that from  definition of tangent cone and the above that  $$(u,\mu,\zeta^\ast,2\vb^T\nabla^2g(\yb)u)\in T_{\Gr N_{\KbG\times\Lb}}(\vb,\lb,\zba,\vb^T\nabla^2g(\yb)\vb)=\Gr N_{\tilde K(\vb,\vba)},$$
where the equation follows from  (\ref{EqTanNormalConePolyhedral}). Thus combining the above inclusion and (\ref{ustar*}), we have that $(u,u^\ast)\in{\cal R}$ and the inclusion $T_{T_{\Gr \widehat N_\Gamma}(\yb,\yba)}(\vb,\vba)\subseteq{\cal R}$ is shown.

{\bf Step 2.} Now we show the reverse inclusion $T_{T_{\Gr \widehat N_\Gamma}(\yb,\yba)}(\vb,\vba)\supseteq{\cal R}$ in \eqref{EqTanConeTanConeGrNormalCone}.
{Let $(u,u^\ast)\in{\cal R}$. Then there exist $\mu, \zeta^\ast$ such that
\begin{eqnarray*}
&&  u^*=\nabla^2(\lb^Tg)(\yb)u+\nabla^2(\mu^Tg)(\yb)\vb+\zeta^\ast,\nonumber \\
&& (u,\mu,\zeta^\ast, 2\vb^T\nabla^2g(\yb)u)\in \Gr N_{\tilde K(\vb,\vba)}=T_{\Gr N_{\KbG\times\Lb}}(\vb,\lb,\zba,\vb^T\nabla^2g(\yb)\vb), 
\end{eqnarray*}
where the second equation follows from  (\ref{EqTanNormalConePolyhedral}).

{First by applying Lemma \ref{LemAux} , we wish to show that
\begin{align} \Theta&:=T_{\{(v,\lambda,z^\ast)\mv \tilde P(v,\lambda,z^\ast)\in \tilde D\}}(\vb,\lb,\zba)=\{(u,\mu,\zeta^\ast)\mv \nabla \tilde P(\vb,\lb,\bar z^\ast)(u,\mu,\zeta^\ast)\in T_{\tilde D}(\tilde P(\vb,\lb,\bar z^\ast))\} \nonumber \\
&= \{(u,\mu,\zeta^\ast)\mv (u,\mu,\zeta^\ast, 2\vb^T\nabla^2g(\yb)u)\in T_{\Gr N_{\KbG\times\Lb}}(\vb,\lb,\zba,\vb^T\nabla^2g(\yb)\vb)\}, \label{Theta}
\end{align}
where
  $\tilde P_1(v,\lambda,z^\ast):=(v,\lambda)$, $\tilde P_2(v,\lambda,z^\ast):=(z^\ast, v^T\nabla^2g(\yb)v)$,   $\tilde D:=\Gr N_{ \KbG\times\Lb}$.
Let   $\zb:=(\vb,\lb,\zba)$. Define
   $L_1:=\big(N_\KbG(\vb)\big)^+\times \R^q\supseteq \big(N_{ \KbG\times\Lb}(\tilde P_1(\zb))\big)^+$ and
  $L_2:=\R^m\times \big(\Lb(\vb)\big)^+$.   Note that
\begin{equation}\label{criticalcone}  \K_{\Lb}(\lb,\vb^T\nabla^2g(\yb)\vb)=T_{\bar\Lambda}(\bar{\lambda}) \cap [\vb^T\nabla^2g(\yb)\vb]^\perp=T_{\Lb(\vb)}(\lb),\end{equation}
where the second equality follows from the fact that $\mu \in T_{\bar\Lambda}(\bar{\lambda}) \cap [\vb^T\nabla^2g(\yb)\vb]^\perp$ if and only if $ \bar{\lambda} +\alpha \mu  \in \bar{\Lambda}(\vb)$ for all $\alpha\geq 0$ sufficiently small. It follows together with (\ref{affinehull}) }
 that
$$L_2:=\R^m\times \big(\Lb(\vb)\big)^+\supseteq \big(\K_{ \KbG\times\Lb}(\tilde P_1(\zb),\tilde P_2(\zb))\big)^+.$$ } Next consider $(w^\ast,\lambda^\ast,z,\mu)\in L_1\times L_2$ satisfying
\[\nabla \tilde P(\zb)^T(w^\ast, \lambda^\ast,z, \mu)=(w^\ast+ 2\nabla^2(\mu^Tg)(\yb)\vb, \lambda^\ast,z)=(0,0,0).\]
Then $\nabla^2(\mu^Tg)(\yb)\vb=-w^\ast/2\in \big(N_{\KbG}(\vb)\big)^+$ and by the assumed 2-nondegeneracy we obtain $\mu=0$ and consequently $w^\ast=0$. Because we also have $\lambda^\ast=0$ and $z=0$, \eqref{EqRelaxedMo} is verified and by \eqref{EqTanPropAux} we obtain (\ref{Theta}).

It follows from (\ref{Theta}) that $(u,\mu, \zeta^*) \in \Theta$. Consequently by the  definition of the tangent cone $T_{\{(v,\lambda,z^\ast)\mv \tilde P(v,\lambda,z^\ast)\in \tilde D\}}(\vb,\lb,\zba)$,   there exist sequences $t_k\downarrow 0$ and $(u_k,\mu_k,\zeta_k^\ast)\to (u,\mu,\zeta^\ast)$ such that
\begin{align*}\lefteqn{\tilde P(\vb+t_ku_k, \lb+t_k\mu_k,\zba+t_k\zeta_k^\ast)}\\
&=(\vb+t_ku_k, \lb+t_k\mu_k,\zba+t_k\zeta_k^\ast, (\vb+t_ku_k)^T\nabla^2g(\yb)(\vb+t_ku_k))\in \Gr N_{\KbG\times\Lb}.
\end{align*}
By \eqref{EqTanConeGrNormalConeAlt} it follows that $(\vb+t_ku_k, \nabla^2\big((\lb+t_k\mu_k)^Tg\big)(\yb)(\vb+t_ku_k)+\zba+t_k\zeta_k^\ast)\in T_{\Gr \widehat N_\Gamma}(\yb,\yba)$ implying $(u,u^\ast)\in T_{T_{\Gr \widehat N_\Gamma}(\yb,\yba)}(\vb,\vba)$. Hence \eqref{EqTanConeTanConeGrNormalCone}  is shown.
\if{ Let $(u,u^\ast)\in{\cal R}$. Then there exist $\mu, \zeta^\ast$ such that
\begin{eqnarray}
&&  u^*=\nabla^2(\lb^Tg)(\yb)u+\nabla^2(\mu^Tg)(\yb)\vb+\zeta^\ast,\nonumber \\
&& (u,\mu,\zeta^\ast, 2\vb^T\nabla^2g(\yb)u)\in \Gr N_{\tilde K(\vb,\vba)}=T_{\Gr N_{\KbG\times\Lb}}(\vb,\lb,\zba,\vb^T\nabla^2g(\yb)\vb), \label{u^*}
\end{eqnarray}
where the second equation follows from  (\ref{EqTanNormalConePolyhedral}).

It follows from (\ref{u^*})   and the  definition of the tangent cone that   there exist sequences $t_k\searrow 0$ and $(u_k,\mu_k,\zeta_k^\ast)\to (u,\mu,\zeta^\ast)$ such that
\[(\vb+t_ku_k, \lb+t_k\mu_k,\zba+t_k\zeta_k^\ast, (\vb+t_ku_k)^T\nabla^2g(\yb)(\vb+t_ku_k))\in \Gr N_{\KbG\times\Lb}.\]
By \eqref{EqTanConeGrNormalConeAlt} it follows that $(\vb+t_ku_k, \nabla^2\big((\lb+t_k\mu_k)^Tg\big)(\yb)(\vb+t_ku_k)+\zba+t_k\zeta_k^\ast)\in T_{\Gr \widehat N_\Gamma}(\yb,\yba)$ implying $(u,u^\ast)\in T_{T_{\Gr \widehat N_\Gamma}(\yb,\yba)}(\vb,\vba)$. Hence \eqref{EqTanConeTanConeGrNormalCone}  is shown.

To show \eqref{EqRegNormalConeTanConeGrNormalCone}, first by applying Lemma \ref{LemAux} , we wish to show that
\begin{align} \Theta&:=T_{\{(v,\lambda,z^\ast)\mv \tilde P(v,\lambda,z^\ast)\in \tilde D\}}(\vb,\lb,\zba)=\{(u,\mu,\zeta^\ast)\mv \nabla \tilde P(\vb,\lb,\bar z^\ast)(u,\mu,\zeta^\ast)\in T_{\tilde D}(\tilde P(\vb,\lb,\bar z^\ast))\} \nonumber \\
&= \{(u,\mu,\zeta^\ast)\mv (u,\mu,\zeta^\ast, 2\vb^T\nabla^2g(\yb)u)\in T_{\Gr N_{\KbG\times\Lb}}(\vb,\lb,\zba,\vb^T\nabla^2g(\yb)\vb)\}, \label{Theta}
\end{align}
where
  $\tilde P_1(v,\lambda,z^\ast):=(v,\lambda)$, $\tilde P_2(v,\lambda,z^\ast):=(z^\ast, v^T\nabla^2g(\yb)v)$,   $\tilde D:=\Gr N_{ \KbG\times\Lb}$.
Let   $\zb:=(\vb,\lb,\zba)$. Define
   $L_1:=\big(N_\KbG(\vb)\big)^+\times \R^q\supset \big(N_{ \KbG\times\Lb}(\tilde P_1(\zb))\big)^+$ and
  $L_2:=\R^m\times \big(\Lb(\vb)\big)^+$.   Since
\begin{equation}\label{criticalcone}  \K_{\Lb}(\lb,\vb^T\nabla^2g(\yb)\vb)=T_{\bar\Lambda}(\bar{\lambda}) \cap [\vb^T\nabla^2g(\yb)\vb]^\perp=T_{\Lb(\vb)}(\lb),\end{equation}
where the second equality follows from the fact that $\mu \in T_{\bar\Lambda}(\bar{\lambda}) \cap [\vb^T\nabla^2g(\yb)\vb]^\perp$ if and only if $ \bar{\lambda} +\alpha \mu  \in \bar{\Lambda}(v)$ for all $\alpha\geq 0$,
we have that $L_2:=\R^m\times \big(\Lb(\vb)\big)^+\supset\big(\K_{ \KbG\times\Lb}(\tilde P_1(\zb),\tilde P_2(\zb))\big)^+$. Next consider $(w^\ast,\lambda^\ast,z,\mu)\in L_1\times L_2$ satisfying
\[\nabla \tilde P(\zb)^T(w^\ast, \lambda^\ast,z, \mu)=(w^\ast+ 2\nabla^2(\mu^Tg)(\yb)\vb, \lambda^\ast,z)=(0,0,0).\]
Then $\nabla^2(\mu^Tg)(\yb)\vb=-w^\ast/2\in \big(N_{\KbG}(\vb)\big)^+$ and by the assumed 2-nondegeneracy we obtain $\mu=0$ and consequently $w^\ast=0$. Because we also have $\lambda^\ast=0$ and $z=0$, \eqref{EqRelaxedMo} is verified and by \eqref{EqTanPropAux} we obtain (\ref{Theta}).}\fi

{\bf Step 3}. {To show \eqref{EqRegNormalConeTanConeGrNormalCone}, note that}
\begin{eqnarray}\lefteqn{\widehat N_{T_{\Gr \widehat N_\Gamma}(\yb,\yba)}(\vb,\vba)} \nonumber \\
&=&  (T_{T_{\Gr \widehat N_\Gamma}(\yb,\yba)}(\vb,\vba))^\circ \nonumber\\
&=& (\{ (u, u^*): u^*= \nabla^2(\lb^Tg)(\yb)u+\nabla^2(\mu^Tg)(\yb)\vb+\zeta^\ast, (u,\mu, \zeta^*)\in \Theta\})^\circ \nonumber \\
&=&
\{(w^\ast,w)\mv \skalp{w^\ast, u}+\skalp{w,\nabla^2(\lb^Tg)(\yb)u+\nabla^2(\mu^Tg)(\yb)\vb+\zeta^\ast}\leq 0  \quad \forall (u,\mu,\zeta^\ast)\in\Theta\} \nonumber\\
&=&\{(w^\ast,w)\mv (w^\ast+\nabla^2(\lb^Tg)(\yb)w, \vb^T\nabla^2 g(\yb)w,w)\in\Theta^\circ\}, \label{NormalC}
\end{eqnarray}
where the second equality follows from  (\ref{EqTanConeTanConeGrNormalCone}).
By \eqref{EqNormalPropAux} together with \eqref{EqNormalConePolyhedral} we have
\begin{eqnarray*}
\Theta^\circ& =& (T_{\{(v,\lambda,z^\ast)\mv \tilde P(v,\lambda,z^\ast)\in \tilde D\}}(\vb,\lb,\zba))^\circ\\
&=& \widehat{N}_{\{(v,\lambda,z^\ast)\mv \tilde P(v,\lambda,z^\ast)\in \tilde D\}}(\vb,\lb,\zba)) \\
&=& \nabla \tilde{P}(\vb,\lb, \zba)^T \big(\tilde K(\vb,\vba)\big)^\circ\times \tilde K(\vb,\vba)\\
&=&
\{(v^\ast+2\nabla^2(\eta^Tg)(\yb)\vb,\xi,v)\mv (v^\ast,\xi, v,\eta)\in \big(\tilde K(\vb,\vba)\big)^\circ\times \tilde K(\vb,\vba)\}\end{eqnarray*}
and \eqref{EqRegNormalConeTanConeGrNormalCone} follows from (\ref{NormalC}).
\end{proof}
Unless $\Lb$ is a singleton, $g$ can not be 2-nondegenerate in direction $\vb=0$. Hence, Proposition \ref{PropTanConeTanConeGrNormalCone} might not be useful in case when $\vb=0$ and $\Lb$ contains more than one element. We now want to cover this situation.
We denote for every $\vb\in \KbG$, $\vba\in N_{\KbG}(\vb)$ by $\Sigma(\vb,\vba)$ a nonempty subset of the extreme points  of $\Lb(\vb)$ such that for every direction $u\in\K_{\KbG}(\vb,\vba)$ we have
\[\Sigma(\vb,\vba)\cap \Lb(\vb+\beta u)\not=\emptyset\ \mbox{for all $\beta>0$ sufficiently small.}\]
We can always choose  $\Sigma(\vb,\vba)$ as the collection of all extreme points of $\Lb(\vb)$, because by \cite[Lemma 3.5]{Rob84} we have $\Lb(v)\subseteq\Lb(\vb)$ for every $v$ sufficiently close to $\vb$ and the set $\Lb(v)$ is a face of $\Lb(\vb)$ whose extreme points are also extreme points of $\Lb(\vb)$. However, it might be advantageous to choose $\Sigma(\vb,\vba)$ smaller to get a sharper inclusion in the following proposition.

\begin{proposition}\label{PropRegNormalConeTanConeGrNormalConeZero}
Let $\vba\in \KbG^\circ$. Then
\begin{eqnarray}\label{EqRegNormalConeTanConeGrNormalConeZero}\lefteqn{\widehat N_{T_{\Gr \widehat N_\Gamma}(\yb,\yba)}(0,\vba)}\\
\nonumber&&\hskip-2em\subseteq\bigcap_{\vb\in \Lsp(\KbG)}\Big\{(w^\ast,w)\mv \exists \lb\in\co\Sigma(\vb,\vba):\ (w^\ast+\nabla^2(\lb^Tg)(\yb)w,w)\in \big(\K_{\KbG}(0,\vba)\big)^\circ\times \K_{\KbG}(0,\vba)\Big\}.\end{eqnarray}
Moreover, for every $\vb\in \Lsp(\KbG)$ such that $g$ is 2-nondegenerate at $(\yb,\yba)$ in direction $\vb$ we have
\begin{equation}\label{EqRegNormalConeTanConeGrNormalConeZero1}\widehat N_{T_{\Gr \widehat N_\Gamma}(\yb,\yba)}(0,\vba)\subseteq \bigcap_{\lb\in\Lb(\vb)}\widehat N_{T_{\Gr \widehat N_\Gamma}(\yb,\yba)}(\vb,\nabla^2(\lb^Tg)(\yb)\vb+\vba).\end{equation}
\end{proposition}
\begin{proof}
 Let $(w^\ast,w)\in \widehat N_{T_{\Gr \widehat N_\Gamma}(\yb,\yba)}(0,\vba)$ and let $\vb\in \Lsp(\KbG)=\KbG\cap (-\KbG)$ be arbitrarily fixed.  We first show that $w\in \K_{\KbG}(0,\vba)$.
For every $z^\ast\in T_{\KbG^\circ}(\vba)$ we have 
${\vba+\alpha z^\ast}\in \KbG^\circ=N_\KbG(0)$ for $\alpha> 0$ small enough. But by
(\ref{EqTanConeGrNormalCone}) with $v=0$, we have $(0,\vba+\alpha z^\ast)\in T_{\Gr \widehat N_\Gamma}(\yb,\yba)$  and thus $\skalp{w^\ast, 0}+\skalp{w,\vba+\alpha z^\ast-\vba}\leq 0$ implying $w\in \big(T_{\KbG^\circ}(\vba){\big)}^\circ=N_{\KbG^\circ}(\vba)=\K_{\KbG}(0,\vba)$.

 Next we show that there exists $\lb\in\co\Sigma(\vb,\vba)$ such that
 \begin{equation} \label{second}
 w^\ast+\nabla^2(\lb^Tg)(\yb)w\in \big(\K_{\KbG}(0,\vba)\big)^\circ.
 \end{equation}
    Note that  $-\vb\in\KbG$, $\Lb(\vb)=\Lb(-\vb)$, and since $\KbG$ is a convex polyhedral cone,
 $$\vba\in \KbG^\circ=T_{\KbG}(\bar v)^\circ=N_{\KbG}(\vb)=N_{\KbG}(-\vb).$$  Moreover by (\ref{EqChar_lambda}),
$$ \lambda\in\Lb (\bar v)\Leftrightarrow \bar v^T\nabla^2g(\yb) \bar v\in N_{\Lb}(\lambda). $$ Therefore by (\ref{EqTanConeGrNormalConeAlt}), $(\pm \alpha\vb, \pm\alpha\nabla^2(\lambda^Tg)(\yb)\vb+\vba)\in T_{\Gr \widehat N_\Gamma}(\yb,\yba)$, $\forall\alpha> 0$ sufficiently small, $\forall \lambda\in\Lb(\vb)$.
 By the definition of the regular normal cone we conclude
\[ \limsup_{\alpha \searrow 0}\frac{\skalp{w^\ast,\pm\alpha\vb}+\skalp{w,\pm\alpha\nabla^2(\lambda^Tg)(\yb)\vb+\vba-\vba}}\alpha=\pm\big(\skalp{w^\ast,\vb}+\skalp{w,\nabla^2(\lambda^Tg)(\yb)\vb}\big)\leq0\]
and therefore
\begin{equation}\label{eqn61}
\skalp{w^\ast,\vb}+\skalp{w,\nabla^2(\lambda^Tg)(\yb)\vb}=0\ \quad \forall \lambda\in\Lb(\vb).
\end{equation}
Consider $u\in\K_{\KbG}(\vb,\vba)$ {and} choose $\beta>0$ sufficiently small such $\Sigma(\vb,\vba)\cap \Lb(\vb+\beta u)\not=\emptyset$.   Then $u\in T_\KbG(\vb)$ and $u^T \vba=0$. It follows that $\vb+\beta u\in \KbG$ for $\beta>0$ small and hence $\langle \vba, \vb +\beta u \rangle =0$ due to the fact that $\vba \in N_\KbG(\vb)$. Hence $\vba \in N_\KbG( \vb+\beta u)$. Let  $\lambda\in\Sigma(\vb,\vba)\cap \Lb(\vb+\beta u)$ and  $\alpha>0$. Since $\lambda\in  \Lb(\alpha(\vb+\beta u))$ and $\vba \in N_{\KbG}(\vb+\beta u) $, by (\ref{EqTanConeGrNormalCone})  we have $$\big(\alpha(\vb+\beta u),\alpha\nabla^2(\lambda^Tg)(\yb)(\vb+\beta u)+\vba\big)\in T_{\Gr \widehat N_\Gamma}(\yb,\yba). $$ It follows by definition for  the regular normal  cone $\widehat N_{T_{\Gr \widehat N_\Gamma}(\yb,\yba)}(0,\vba)$ that
\begin{eqnarray*}
\lefteqn{ \limsup_{\alpha\downarrow 0}\frac{\skalp{w^\ast,\alpha(\vb+\beta u)}+\skalp{w, \alpha\nabla^2(\lambda^Tg)(\yb)(\vb+\beta u)+\vba-\vba}}{\alpha}}\\
&&=\beta\big(\skalp{w^\ast,u}+\skalp{w,\nabla^2(\lambda^Tg)(\yb)u}\big)\leq 0,\end{eqnarray*}
where the equality follows from (\ref{eqn61}).
Hence $$\skalp{w^\ast,u}+\skalp{w,\nabla^2(\lambda^Tg)(\yb)u}\leq 0 \quad \forall u\in\K_{\KbG}(\vb,\vba) , \lambda\in\Sigma(\vb,\vba)$$ and by taking into account that $\co \Sigma(\vb,\vba)$ is compact as the convex hull of a finite set, we obtain
\begin{eqnarray*}0&\geq& \max_{u\in\K_{\KbG}(\vb,\vba)\cap\B_{\R^m}}\min_{\lambda\in \co\Sigma(\vb,\vba)}\skalp{w^\ast,u}+\skalp{w,\nabla^2(\lb^Tg)(\yb)u}\\
&=&\min_{\lambda\in \co\Sigma(\vb,\vba)}\max_{u\in\K_{\KbG}(\vb,\vba)\cap\B_{\R^m}}\skalp{w^\ast,u}+\skalp{w,\nabla^2(\lb^Tg)(\yb)u}.\end{eqnarray*}
Hence there is $\lb\in \co\Sigma(\vb,\vba)$ such that $\max_{u\in\K_{\KbG}(\vb,\vba)}\skalp{w^\ast,u}+\skalp{w,\nabla^2(\lb^Tg)(\yb)u}\leq 0$.
Since $\vb\in \Lsp(\KbG)$, we have $T_{\KbG}(\vb)
=\KbG$ and $\K_{\KbG}(\vb,\vba)=\K_{\KbG}(0,\vba)$.
 Therefore (\ref{second}) holds. Putting all together, \eqref{EqRegNormalConeTanConeGrNormalConeZero} follows.

Let $\vb\in \Lsp(\KbG)$. We now show \eqref{EqRegNormalConeTanConeGrNormalConeZero1} under the assumption  that $g$ is 2-nondegenerate in direction $\vb$ at $(\yb,\yba)$. Let $ (w^*, w)\in \widehat N_{T_{\Gr \widehat N_\Gamma}(\yb,\yba)}(0,\vba)$.  Fixing $\lb\in\Lb(\vb)$, we wish to prove that
$ (w^*, w)\in \widehat N_{T_{\Gr \widehat N_\Gamma}(\yb,\yba)}(\vb,\nabla^2(\lb^Tg)(\yb)\vb+\vba)=(T_{T_{\Gr \widehat N_\Gamma}(\yb,\yba)}(\vb, \nabla^2 (\lb^Tg)(\yb)\vb+\vba))^\circ$.   So  consider $(u,u^\ast)\in T_{T_{\Gr \widehat N_\Gamma}(\yb,\yba)}(\vb, \nabla^2 (\lb^Tg)(\yb)\vb+\vba)$. By Proposition \ref{PropTanConeTanConeGrNormalCone} there are elements $\mu,\zeta^\ast$ such that
\begin{eqnarray*}
&&  u^\ast=\nabla^2(\lb^Tg)(\yb)u+\nabla^2(\mu^Tg)(\yb)\vb+\zeta^\ast,\\
&&
(u,\mu,\zeta^\ast, 2\vb^T\nabla^2g(\yb)u)\in  {\Gr N_{\tilde K(\vb,\nabla^2 (\lb^Tg)(\yb)\vb+\vba)}=} T_{\Gr N_{\KbG\times\Lb}}(\vb,\lb,
{\vba} ,\vb^T\nabla^2g(\yb)\vb).
\end{eqnarray*}
{By taking into account \eqref{Theta},}   there are sequences $t_k\downarrow 0$, $(u_k,\mu_k,\zeta_k^\ast)\to (u,\mu,\zeta^\ast)$
such that  for each $k$,
$\vba+t_k\zeta_k^\ast\in N_{\KbG}(\vb+t_ku_k), \ {(\bar v+t_k u_k)^T\nabla^2 g(\yb)(}\bar v+t_k u_k)\in N_{\bar{\Lambda}}(\bar \lambda +t_k \mu_k)$. Note that by (\ref{EqChar_lambda}), ${(\bar v+t_k u_k)^T\nabla^2 g(\yb)(}\bar v+t_k u_k)\in N_{\bar{\Lambda}}(\bar \lambda +t_k \mu_k)$ if and only if $ \lb+t_k\mu_k\in\Lb(\vb+t_ku_k)$, and so
$$ \lb+t_k\mu_k\in\Lb(\vb+t_ku_k), \quad \vba+t_k\zeta_k^\ast\in N_{\KbG}(\vb+t_ku_k).$$
The set $N_{\KbG}(\vb+t_ku_k)$ is a face of $\KbG^\circ$ and since the polyhedral convex cone $\KbG$ only has finitely many faces, after passing to a subsequence we can assume that $N_{\KbG}(\vb+t_ku_k)=F$ $\forall k$ for some face $F$ of $\KbG^\circ$. Since $F$ is closed, we obtain $\vba\in F$ and thus $\vba+\alpha t_k\zeta_k^\ast\in F=N_{\KbG}(\vb+t_ku_k)$ $\forall k$, $\forall \alpha\in[0,1]$. Hence, for every $k$ and every $\alpha\in[0,1]$ we have $\Big(\alpha(\vb+t_ku_k), \alpha \nabla^2\big((\lb+t_k\mu_k)^Tg\big)(\yb)(\vb+t_ku_k)+ \vba+\alpha t_k^\ast\zeta_k^\ast\Big)\in T_{\Gr \widehat N_\Gamma}(\yb,\yba)$ by (\ref{EqTanConeGrNormalCone}),  implying
\begin{eqnarray*}
0&\geq& \lim_{\alpha\downarrow 0}\frac{\skalp{w^\ast,\alpha(\vb+t_ku_k)}+\skalp{w, \alpha \nabla^2\big((\lb+t_k\mu_k)^Tg\big)(\yb)(\vb+t_ku_k)+ \vba+\alpha t_k\zeta_k^\ast-\vba}}{\alpha}\\
&=&\skalp{w^\ast,(\vb+t_ku_k)}+\skalp{w, \nabla^2\big((\lb+t_k\mu_k)^Tg\big)(\yb)(\vb+t_ku_k)+  t_k\zeta_k^\ast}\\
&=& t_k\big(\skalp{w^\ast,u_k}+\skalp{w, \nabla^2\big((\lb+t_k\mu_k)^Tg\big)(\yb)u_k+ \nabla^2(\mu_k^Tg)(\yb)\vb +  \zeta_k^\ast}\big)\end{eqnarray*}
for all $k$. Dividing by $t_k$ and passing to the limit we obtain
\[0\geq \skalp{w^\ast,u}+\skalp{w, \nabla^2(\lb^Tg)(\yb)u+ \nabla^2(\mu^Tg)(\yb)\vb +  \zeta^\ast}=\skalp{w^\ast,u}+\skalp{w,u^\ast}.\]
Thus $(w^\ast,w)\in (T_{T_{\Gr \widehat N_\Gamma}(\yb,\yba)}(\vb, \nabla (\lb^Tg)(\yb)\vb+\vba))^\circ$ and the inclusion \eqref{EqRegNormalConeTanConeGrNormalConeZero1} follows.
\end{proof}

\subsection{Regular normals to tangents of tangent cones}
Throughout this subsection let $(\vb,\vba)\in T_{\Gr \widehat N_\Gamma}(\yb,\yba)$ and $(\delta\vb,\delta\vba)\in T_{T_{\Gr \widehat N_\Gamma}(\yb,\yba)}(\vb,\vba)$ be given and we assume that $g$ is 2-nondegenerate in direction $\vb$ at $(\yb,\yba)$. Further let $(\lb, \zba)\in\Lb(\vb)\times N_\KbG(\vb)$ denote the unique element fulfilling
\eqref{EqReprVast}, i.e.,
$\vba =\nabla^2(\lb^Tg)(\yb)\vb+\zba$,   and let according to \eqref{EqTanConeTanConeGrNormalCone} $(\bar \mu,\bar\zeta^*)$ denote some element with
\begin{align}
\label{EqBarMu1}  &\delta\vba = \nabla^2(\lb^Tg)(\yb)\delta\vb+\nabla^2(\bar\mu^Tg)(\yb)\vb+\bar\zeta^\ast,\\
\label{EqBarMu2}  &(\delta\vb,\bar \mu,\bar \zeta^\ast, 2\vb^T\nabla^2g(\yb)\delta\vb )\in \Gr N_{\tilde K(\vb,\vba)}=T_{\Gr N_{\KbG\times\Lb}}(\vb,\lb,\zba,\vb^T\nabla^2g(\yb)\vb),
\end{align}
where the equality in (\ref{EqBarMu2}) follows from (\ref{EqTanNormalConePolyhedral}).
Note that by definition, $$\tilde K(\vb,\vba):=\K_{\KbG\times\Lb}(\vb,\lb,\zba,\vb^T\nabla^2g(\yb)\vb)$$ and hence it
 follows that $\bar\mu\in \K_{\Lb}(\lb,\vb^T\nabla^2 g(\yb)\vb)=T_{\Lb(v)}(\lb)$ where the equality follows from (\ref{criticalcone}), $\bar\zeta^\ast\in N_{\K_\KbG(\vb,\zba)}(\delta\vb)\subseteq (\K_\KbG(\vb,\zba))^\circ {\subseteq \big(N_\KbG(\vb)\big)^+.}$ 
By (\ref{EqBarMu1}),  $\A_{\vb}(\bar\mu,\bar\zeta^*)=\delta\vba - \nabla^2(\lb^Tg)(\yb)\delta\vb$ and from Lemma \ref{LemUnique_lambda} we conclude that $(\bar \mu,\bar\zeta^*)$ are unique.


\begin{proposition}\label{PropTanConeTanConeTanConeGrNormalCone}
Under the assumption stated in the beginning of this subsection, we have
\begin{eqnarray}\label{EqTanConTanConeTanConeGrNormalCone}
 \lefteqn{ T_{T_{T_{\Gr \widehat N_\Gamma}(\yb,\yba)}(\vb,\vba)}(\delta \vb,\delta \vba)} \nonumber \\
 && =\left \{ (u, u^*) \mv  \exists \delta \mu, \delta\zeta^\ast : \begin{array}{l}
  u^*=\nabla^2(\lb^Tg)(\yb)u+\nabla^2(\delta\mu^Tg)(\yb)\vb+\delta\zeta^\ast)\\
  (u,\delta\mu,\delta\zeta^\ast, 2\vb^T\nabla^2g(\yb)u)\in \Gr N_{\tilde K(\vb,\vba,\delta\vb,\delta\vba)}\end{array} \right \}
\end{eqnarray}
and
\begin{eqnarray}\label{EqRegNormalConeTanConeTanConeGrNormalCone}
 \lefteqn{ \widehat N_{T_{T_{\Gr \widehat N_\Gamma}(\yb,\yba)}(\vb,\vba)}(\delta \vb,\delta \vba)    }\nonumber \\
 && =\Big\{(w^\ast,w)\mv  \exists \eta:
  \begin{array}{l}  \big(w^\ast+\nabla^2(\lb^Tg)(\yb)w- 2\nabla^2(\eta^Tg)(\yb)\vb, \vb^T\nabla^2 g(\yb)w,w,\eta\big)\\
 \in \big(\tilde K(\vb,\vba,\delta\vb,\delta\vba)\big)^\circ\times \tilde K(\vb,\vba,\delta\vb,\delta\vba)\end{array} \Big\},
\end{eqnarray}
where $\tilde K(\vb,\vba,\delta\vb,\delta\vba):=\K_{\tilde K(\vb,\vba)}(\delta\vb,\bar\mu,\bar\zeta^*,2\vb^T\nabla^2g(\yb)\delta\vb)$.
\end{proposition}
\begin{proof}
   We use similar arguments as in the proof of Proposition \ref{PropTanConeTanConeGrNormalCone}.
   Let ${\cal R}$ denote the set on the right hand side of \eqref{EqTanConTanConeTanConeGrNormalCone} and consider $(u, u^*)\in T_{T_{T_{\Gr \widehat N_\Gamma}(\yb,\yba)}(\vb,\vba)}(\delta \vb,\delta \vba)$ together with sequences $t_k\downarrow 0$ and $(u_k,u_k^*)\to(u,u^*)$ with $(\delta \vb+t_ku_k,\delta \vba+t_k u_k^*)\in T_{T_{\Gr \widehat N_\Gamma}(\yb,\yba)}(\vb,\vba)$.
  {By Proposition \ref{PropTanConeTanConeGrNormalCone} there are  elements $\mu_k$, $\zeta_k^*$ such that}
  \begin{align*}&\delta\vba +t_ku_k^*= \nabla^2(\lb^Tg)(\yb)(\delta \vb+t_ku_k) +\nabla^2(\mu_k^Tg)(\yb)\vb+\zeta_k^\ast,\\
  &(\delta \vb+t_ku_k, \mu_k,\zeta_k^\ast, 2\vb^T\nabla^2g(\yb)(\delta \vb+ t_k u_k)\in \Gr N_{\tilde K(\vb,\vba)}=T_{\Gr N_{\KbG\times\Lb}}(\vb,\lb,\zba,\vb^T\nabla^2g(\yb)\vb),
  \end{align*}
 where the equality in the second inclusion follows from (\ref{EqTanNormalConePolyhedral}). By taking into account \eqref{EqBarMu1} we obtain after rearranging
  \[u_k^*- \nabla^2(\lb^Tg)(\yb)u_k = \nabla^2(\frac{(\mu_k-\bar\mu)^T}{t_k}g)(\yb)\vb+\frac{\zeta_k^\ast-\bar\zeta^*}{t_k}.\]
 Similarly as shown in the paragraph before Proposition \ref{PropTanConeTanConeTanConeGrNormalCone}, we can show that both $\bar  \mu$ and $ \mu_k$  belong to $\K_{\Lb}(\lb,\vb^T\nabla^2 g(\yb)\vb)=T_{\Lb(v)}(\lb)$. Hence  we obtain $\mu_k-\bar\mu\in \big(\Lb(\vb)\big)^+$.  Further, $\bar\zeta^\ast\in N_{\K_\KbG(\vb,\zba)}(\delta\vb)\subseteq {\big(N_\KbG(\vb)\big)^+}$ 
 and $\zeta_k^\ast\in N_{\K_\KbG(\vb,\zba)}(\delta\vb+t_ku_k)\subseteq {\big(N_\KbG(\vb)\big)^+}$ 
  implying $\zeta_k^*-\bar\zeta^*\in \big(N_\KbG(\vb)\big)^+$. Thus, by Lemma \ref{LemUnique_lambda} the sequences  $\frac{\mu_k-\bar\mu}{t_k}$ and $\frac{\zeta_k^\ast-\bar\zeta^*}{t_k}$ converge to some elements $\delta\mu$ and $\delta\zeta^*$, respectively, with $u^*-\nabla^2(\lb^Tg)(\yb)u=\nabla^2(\delta\mu^Tg)(\yb)\vb+\delta \zeta^\ast$ and
  \[(u,\delta\mu,\delta\zeta^*, 2\vb^T\nabla^2g(\yb) u)\in T_{\Gr N_{\tilde K(\vb,\vba)}}(\delta\vb, \bar\mu,\bar\zeta^*, 2\vb^T\nabla^2g(\yb)\delta \vb)=\Gr N_{\tilde K(\vb,\vba,\delta\vb,\delta\vba)}\]
  verifying $(u,u^*)\in {\cal R}$.

  Now we prove the reverse inclusion of \eqref{EqTanConTanConeTanConeGrNormalCone}. Let  $(u,u^*)\in {\cal R}$. Then there  exist   $\delta\mu$ and $\delta\zeta^\ast$ such that
\begin{eqnarray}
&& u^*= \nabla^2(\lb^Tg)(\yb)u+\nabla^2(\delta\mu^Tg)(\yb)\vb+\delta\zeta^\ast, \label{ustarnew}\\
&& (u,\delta\mu,\delta\zeta^\ast, 2\vb^T\nabla^2g(\yb)u)\in \Gr N_{\tilde K(\vb,\vba,\delta\vb,\delta\vba)}\nonumber \\
&& \qquad \qquad =T_{\Gr N_{\tilde K(\vb,\vba)}}(\delta\vb,\bar\mu,\bar\zeta^*,2\vb^T\nabla^2g(\yb)\delta\vb)  \label{TangentG},
\end{eqnarray}
 where the equality in the second inclusion follows from (\ref{EqTanNormalConePolyhedral}) and the notation $\tilde K(\vb,\vba,\delta\vb,\delta\vba):=\K_{\tilde K(\vb,\vba)}(\delta\vb,\bar\mu,\bar\zeta^*,2\vb^T\nabla^2g(\yb)\delta\vb)$.
  Since $\tilde K(\vb,\vba):=\K_{\KbG\times\Lb}(\vb,\lb,\zba,\vb^T\nabla^2g(\yb)\vb)$ is a convex polyhedral set, $\Gr N_{\tilde K(\vb,\vba)}$ is polyhedral, it follows by (\ref{TangentG}) that     $$(\delta\vb +tu, \bar\mu +t\delta\mu,\bar\zeta^*+t\delta \zeta^*, 2\vb^T\nabla^2g(\yb)(\delta \vb+tu))\in \Gr N_{\tilde K(\vb,\vba)}$$ for all $t>0$ sufficiently small.
 By (\ref{EqTanConeTanConeGrNormalCone}) and taking into account  (\ref{EqBarMu1}) and (\ref{ustarnew}), it follows that $$ (\delta \vb+tu,\delta \vba+t u^*)\in T_{T_{\Gr \widehat N_\Gamma}(\yb,\yba)}(\vb,\vba)$$ from which we can conclude $(u, u^*)\in  T_{T_{T_{\Gr \widehat N_\Gamma}(\yb,\yba)}(\vb,\vba)}(\delta \vb,\delta \vba)$. Thus \eqref{EqTanConTanConeTanConeGrNormalCone} is proven.

  In order to show \eqref{EqRegNormalConeTanConeTanConeGrNormalCone} note that by (\ref{EqTanConTanConeTanConeGrNormalCone}), $\widehat N_{T_{T_{\Gr \widehat N_\Gamma}(\yb,\yba)}(\vb,\vba)}(\delta \vb,\delta \vba)$ is the collection of all $(w^*,w)$ fulfilling
  \begin{align*}0&\geq \skalp{w^\ast,u}+    \skalp{w, \nabla^2(\lb^Tg)(\yb)u+\nabla^2(\delta\mu^Tg)(\yb)\vb+\delta\zeta^\ast}\\
   &=\skalp{w^*+\nabla^2(\lb^Tg)(\yb)w,u}+w^T\nabla^2(\delta\mu^Tg)(\yb)\vb+w^T\delta \zeta^*\end{align*}
   for all
  \[
   (u,\delta\mu,\delta\zeta^*) \in\Theta:=\{(u,\delta\mu,\delta\zeta^*)\mv (u,\delta\mu,\delta\zeta^\ast, 2\vb^T\nabla^2g(\yb)u)\in \Gr N_{\tilde K(\vb,\vba,\delta\vb,\delta\vba)}\},\]
   which is the same as $(w^*+\nabla^2(\lb^Tg)(\yb)w,w^T\nabla^2 g(\yb)\vb,w)\in\Theta^\circ$. In order to compute $\Theta^\circ$ we use Lemma \ref{LemAux} with the linear mappings
   $\tilde P_1(u,\delta\mu,\delta\zeta^*):=(u,\delta\mu)$, $\tilde P_2(u,\delta\mu,\delta\zeta^*):=(\delta\zeta^\ast, 2\vb^T\nabla^2g(\yb)u)$ and $C=\tilde K(\vb,\vba)$ and $\zb=(\delta\vb, \bar\mu,\bar\zeta^*)$. Indeed, for the subspaces $L_1,L_2$ defined in the proof of Proposition \ref{PropTanConeTanConeGrNormalCone} we have shown
   $\ker\nabla \tilde P(\zb)\cap (L_1\times L_2)=\{0\}$, where we have to take into account that $\nabla \tilde P$ coincides with the derivative of the mapping $\tilde P$ used in the proof of Proposition \ref{PropTanConeTanConeGrNormalCone} at $(\vb,\lb,\zba)$. Further, from \eqref{EqLspTanGraphNormalC} together with \eqref{EqLspTanGrNormalCritCone} and the definition of $\tilde K(\vb,\vba)$ we obtain
   \begin{align*}L_1^\perp\times L_2^\perp &\subseteq \Lsp\big(T_{\Gr N_{\KbG\times\Lb}}(\vb,\lb,\zba,\vb^T\nabla^2 g(\yb)\vb)\big)\\
   &\subseteq \Lsp\big(T_{\Gr N_{\K_{\KbG\times\Lb}(\vb,\lb,\zba,\vb^T\nabla^2 g(\yb)\vb)}}(\delta\vb, \bar\mu,\bar\zeta^*,2\vb^T\nabla^2 g(\yb)\delta\vb)\big)\\
   &=\Lsp\big(T_{\Gr N_{\tilde K(\vb,\vba)}}(\delta\vb, \bar\mu,\bar\zeta^*,2\vb^T\nabla^2 g(\yb)\delta\vb).
   \end{align*}
   Applying \eqref{EqLspTanGraphNormalC} once more we obtain $$L_1\supseteq  N_{\tilde K(\vb,\vba)}(\delta\vb, \bar\mu), \quad  L_2 \supseteq \K_{\tilde K(\vb,\vba)})(\delta\vb, \bar\mu,\bar\zeta^*,2\vb^T\nabla^2 g(\yb)\delta\vb).$$ Hence we can apply Lemma \ref{LemAux} to obtain
   \begin{eqnarray*}
   \lefteqn{\widehat N_{\{z\mv \tilde P(z)\in\tilde D\}}(\zb)}\\
   &&=\{w\mv \nabla \tilde P(\zb)w\in T_{\tilde D}(\tilde P(\zb))\}^\circ=
   \{w\mv \nabla \tilde P(\zb)w\in \Gr N_{\tilde K(\vb,\vba,\delta\vb,\delta\vba)}\}^\circ=\Theta^\circ\\
   &&=    \{(v^*+2\nabla^2(\eta^Tg)(\yb)\vb,\xi,v)\mv (v^*,\xi,v,\eta)\in \tilde K(\vb,\vba,\delta\vb,\delta\vba)^\circ\times \tilde K(\vb,\vba,\delta\vb,\delta\vba)\},\end{eqnarray*}
   where the second equality follows from (\ref{EqTanNormalConePolyhedral}),
   and hence \eqref{EqRegNormalConeTanConeTanConeGrNormalCone} follows.
\end{proof}

\section{New optimality condition for (MPEC)}
To establish the main optimality condition in Theorem \ref{ThOptCond}, we first apply Theorem \ref{ThLinM_StatBasic} to  problem \eqref{EqMPEC} to obtain the following lemma.
\begin{lemma}\label{LemOptCond}
  Assume that $(\xb,\yb)$ is a local minimizer for  problem \eqref{EqMPEC} fulfilling Assumption \ref{Ass1}. Further assume that $g$ is 2-nondegenerate in every nonzero critical  direction $0\not=v\in\KbG$ at $(\yb,\yba)$ with $\yba:=-\phi(\xb,\yb)$.  Then
  there are a direction $(\delta x,\delta y)$ and elements
  \begin{equation*}
  (\delta\vb,\delta\vba)\in T_{T_{\Gr \widehat N_\Gamma}(\yb,\yba)}\big(\delta y, -\nabla \phi(\xb,\yb)(\delta x,\delta y)\big),\ \  \delta a \in  T_{T_{\R^m_-}(G(\xb,\yb))}\big(\nabla G(\xb,\yb)(\delta x,\delta y)\big),\end{equation*}
   together with multipliers
  \begin{equation*}(w^\ast,w)\in \widehat N_{T_{T_{\Gr \widehat N_\Gamma}(\yb,\yba)}\left(\delta y, -\nabla \phi(\xb,\yb)(\delta x,\delta y)\right)}(\delta\vb,\delta\vba),\
  \sigma\in N_{T_{T_{\R^m_-}(G(\xb,\yb))}\left(\nabla G(\xb,\yb)(\delta x,\delta y)\right)}(\delta a)\end{equation*}
  such that
  \begin{subequations}
  \begin{align}
\label{EqAuxOptCond_a}    &\nabla F(\xb,\yb){^T}(\delta x,\delta y)=0,\\
 \label{EqAuxOptCond_b}   &\nabla_x F(\xb,\yb)-\nabla_x\phi(\xb,\yb)^Tw+\nabla_xG(\xb,\yb)^T\sigma=0,\\
\label{EqAuxOptCond_c}    &\nabla_y F(\xb,\yb)+w^*-\nabla_y\phi(\xb,\yb)^Tw+\nabla_yG(\xb,\yb)^T\sigma=0,\\
\label{EqAuxOptCond_d}    &\delta y=0\ \Rightarrow\ \delta x=0,\\
\label{EqAuxOptCond_e}    &\delta y\not=0\ \Rightarrow\ T_{\Gr\widehat N_\Gamma}(\yb,\yba)\mbox{ is not locally polyhedral near $\big (\delta y,-\nabla \phi(\xb,\yb)(\delta x,\delta y) \big )$}.
  \end{align}
  \end{subequations}
  Further, if $\Lsp(\KbG)\not=\{0\}$ then $(\delta y,\delta\vb)\not=(0,0)$. Otherwise, if $\Lsp(\KbG)=\{0\}$ and $(\delta y,\delta\vb)=(0,0)$ then there is some $\lb\in \Sigma(0,\delta \vba)$ such that
  \begin{equation}  \label{EqAuxOptCondSuppl}
        (w^\ast+\nabla^2(\lb^Tg)(\yb)w,w)\in \big(\K_{\KbG}(0,\delta\vba)\big)^\circ\times \K_{\KbG}(0,\delta\vba).
  \end{equation}
\end{lemma}
\begin{proof}
Let $z:=(x,y), \zb=(\xb,\yb), P(x,y):=\left(\begin{array}{c}(y,-\phi(x,y))\\G(x,y)\end{array}\right),  D:=\Gr \widehat N_\Gamma\times \R^p_-$. Assumption \ref{Ass1} ensures that Theorem \ref{ThLinM_StatBasic} is applicable and so one of  Theorem \ref{ThLinM_StatBasic}(i) and Theorem \ref{ThLinM_StatBasic}(ii) holds.

 If Theorem \ref{ThLinM_StatBasic}(i)  is fulfilled, then there exists a direction $\omega=(\delta\vb,\delta\vba,\delta a)\in T_D(P(\zb))=T_{\Gr \widehat N_\Gamma}(\yb,\yba)\times T_{\R^m_-}(G(\xb,\yb))$ and a multiplier $$\omega^*=(w^*,w,\sigma)\in \widehat N_{T_D(P(\zb))}(w)=\widehat N_{T_{\Gr \widehat N_\Gamma}(\yb,\yba)}(\delta \vb,\delta\vba)\times N_{T_{\R^m_-}(G(\xb,\yb))}(\delta a)$$ such that
$ 0=\nabla F(\xb,\yb)+\nabla P(\xb,\yb)^T \omega^*.$  By virtue of (\ref{normaltangent}),  we see that the conditions \eqref{EqAuxOptCond_a}-\eqref{EqAuxOptCond_c}  are fulfilled with $\delta x=0$, $\delta y=0$.
 Otherwise  Theorem \ref{ThLinM_StatBasic}(ii)  is fulfilled, i.e.,  there is a direction $\bar u=(\delta x,\delta y)$ with
\begin{eqnarray}
\lefteqn{\nabla P(\zb)\bar u =\big(\delta y,-\nabla \phi(\xb,\yb)(\delta x,\delta y),\nabla G(\xb,\yb)(\delta x,\delta y)\big)} \label{deltay}\\
&&\in T_D\big(P(\zb)\big)=T_{\Gr \widehat N_\Gamma}(\yb,\yba)\times T_{\R^p_-}(G(\xb,\yb))\nonumber\end{eqnarray}
 fulfilling \eqref{EqNotZero}, $\nabla F(\xb,\yb)(\delta x,\delta y)=0$ which is (\ref{EqAuxOptCond_a}) and  (\ref{EqKKTNormal}) such that $T_D(P(\zb))$ is not locally polyhedral near $\nabla P(\zb)\bar u$, which is equivalent to the requirement that $T_{\Gr\widehat N_\Gamma}(\yb,\yba)$ is not locally polyhedral near $(\delta y,-\nabla \phi(\xb,\yb)(\delta x,\delta y))$ due to the polyhedrality of $T_{\R^p_-}\big(G(\xb,\yb)\big)$. From \eqref{EqTanConeGrNormalConeAlt} we see that $T_{\Gr \widehat N_\Gamma\times \R^p_-}((\yb,\yba),G(\xb,\yb))$ is the graph of a set-valued mapping $M=M_c+M_p$, where $M_p(v):=N_{\KbG}(v)\times T_{\R^p_-}(G(\xb,\yb))$ is polyhedral and $$M_c(v):=\{\nabla^2(\lambda^Tg)(y)v\mv \lambda\in\Lbv\cap \kappa\norm{\yba} \B_{\R^q}\}\times\{0\}$$ fulfills \eqref{EqBoundM_C}. Further, the graphs of $M_p$ and $M_c$ are closed cones and from \eqref{EqKKT_W_NotZero} we conclude that $\delta y\not=0$ by taking account of (\ref{deltay}).
 {Next we utilize \eqref{EqKKT}, which says $\nabla F(\zb)^T\bar u=0$. Since by the assumed GGCQ we have \[\nabla F(\zb)^Tu\geq 0 \ \forall u\ \mbox{s.t.}\ \nabla P(\zb)u\in T_D(P(\zb)),\]\\
 $\ub$ is a global minimizer of  problem
 $$ \min \nabla   F(\zb){^T}u \quad \mbox{ subject to } \nabla P(\zb)u\in T_D(P(\zb)). $$}
  \if{Next we utilize \eqref{EqKKTNormal}, which says $$0\in \nabla F(\zb)+\widehat N_{\{u\mv\nabla P(\zb)u\in T_D(P(\zb))\}}(\ub).$$

 By \cite[Theorems 6.12, 6.11]{RoWe98}, there exists a continuously differentiable function $\widetilde F : \R^{n+m} \rightarrow \R$ with  $\nabla \widetilde{F}(\ub)=  \nabla F(\zb)$ such that $\ub$ is a global minimizer of  problem
 $$ \min \widetilde  F(u) \quad \mbox{ subject to } \nabla P(\zb)u\in T_D(P(\zb)). $$}\fi
Similarly as in Remark \ref{RemSubReg}(ii),  we can apply Theorem \ref{ThLinM_StatBasic} once more to the above problem, because metric subregularity of $u\rightrightarrows \nabla P(\zb)u-T_D(P(\zb))$ at $(0,0)$ implies metric subregularity at $(\bar u,0)$ by \cite[Lemma 3]{Gfr18} and therefore also GGCQ for the system $\nabla P(\zb)u\in T_D(P(\zb))$ at $\bar u$. This means that the set $\Omega=\{z\mv P(z)\in D\}$ is replaced by the set $\{u\mv\nabla P(\zb)u\in T_D\big(P(\zb)\big)\}$, whose linearized tangent cone at $\ub$ is $\{u\mv \nabla P(\zb)u\in T_{T_D(P(\zb))}(\nabla P(\zb)\ub)\}$. Since
\[T_{T_D(P(\zb))}(\nabla P(\zb)\ub)=T_{T_{\Gr \widehat N_\Gamma}(\yb,\yba)}\big(\delta y, -\nabla \phi(\xb,\yb)(\delta x,\delta y)\big)\times T_{T_{\R^m_-}(G(\xb,\yb))}\big(\nabla G(\xb,\yb)(\delta x,\delta y)\big)\]
and $g$ is 2-nondegenerate in direction $\delta y\not=0$, by \eqref{EqTanConeTanConeGrNormalCone} the set $T_{T_D(P(\zb))}(\nabla P(\zb)\ub)$ is polyhedral and therefore only the first alternative of Theorem \ref{ThLinM_StatBasic} is possible. Hence there is a direction $\omega=(\delta \vb,\delta \vba,\delta a)\in T_{T_D(P(\zb))}(\nabla P(\zb)\ub)$ and a multiplier
\begin{align*}\omega^*&=(w^*,w,\sigma)\in \widehat N_{T_{T_D(P(\zb))}(\nabla P(\zb)\ub)}(\omega)\\
&= \widehat N_{T_{T_{\Gr \widehat N_\Gamma}(\yb,\yba)}\left(\delta y, -\nabla \phi(\xb,\yb)(\delta x,\delta y)\right)}(\delta\vb,\delta\vba)\times N_{T_{T_{\R^m_-}(G(\xb,\yb))}\left(\nabla G(\xb,\yb)(\delta x,\delta y)\right)}(\delta a)\end{align*}
with $0\in \nabla F(\zb)+\nabla P(\zb)^T\omega^*$ which results in \eqref{EqAuxOptCond_b} and  \eqref{EqAuxOptCond_c}.

Now consider the case when $\delta y=0$. In this case we must have  $\delta x=0$. Then by Proposition \ref{closedcone},
\[(w^\ast,w)\in \widehat N_{T_{T_{\Gr \widehat N_\Gamma}(\yb,\yba)}\left(0, 0\right)}(\delta\vb,\delta\vba)= \widehat N_{T_{\Gr \widehat N_\Gamma}(\yb,\yba)}(\delta\vb,\delta\vba).\]
If $\delta \vb=0$ and $\Lsp(\KbG)\not=\{0\}$, then by \eqref{EqRegNormalConeTanConeGrNormalConeZero1} we also have
\[(w^\ast,w)\in \widehat N_{T_{\Gr \widehat N_\Gamma}(\yb,\yba)}\big (\vb,\nabla^2(\lb^Tg)(\yb)\vb+\delta\vba \big)\]
for every $0\not=\vb\in \Lsp(\KbG)\not=\{0\}$ and every $\lb\in\Lb(\vb)$ and therefore we can assume $\delta\vb\not=0$. Otherwise, if  $\delta \vb=0$ and $\Lsp(\KbG)=\{0\}$ then \eqref{EqAuxOptCondSuppl} follows from \eqref{EqRegNormalConeTanConeGrNormalConeZero} by taking $\vb=0$.
\end{proof}

Now we are ready to state and prove our main optimality condition for  problem \eqref{EqMPEC}. The main task is to interpret the formulas for the tangent cones and the regular normal cones in Propositions \ref{PropTanConeTanConeGrNormalCone}-\ref{PropTanConeTanConeTanConeGrNormalCone} appearing in Lemma \ref{LemOptCond} in terms of problem data.
\begin{theorem}
  \label{ThOptCond}   Assume that $(\xb,\yb)$ is a local minimizer for  problem \eqref{EqMPEC} fulfilling Assumption \ref{Ass1}. Further assume that $g$ is 2-nondegenerate in every nonzero critical  direction $0\not=v\in\KbG$ at $(\yb,\yba)$, where $\yba:=-\phi(\xb,\yb)$.  Then
  there are  $\vb\in\KbG$,  $\zba\in N_{\KbG}(\vb)$, $\lb\in \Lb(\vb)$,
  two faces $\F^v_1,\F^v_2$ of $\K_{\KbG}(\vb,\zba)$ with $\F^v_2\subseteq\F^v_1$, $\delta v\in\ri \F^v_2$, two faces $\F^\lambda_1,\F^\lambda_2$ of $T_{\Lb(\vb)}(\lb)$ with $\F^\lambda_2
  \subseteq \F^\lambda_1$, $w\in \F^v_1-\F^v_2$, $\eta\in \F^\lambda_1-\F^\lambda_2$ and  $\sigma\in\R^p_+$ such that
  \begin{subequations}
  \label{EqOptCond}
  \begin{align}
  \label{EqOptCond_a}&\nabla_x F(\xb,\yb)-\nabla_x\phi(\xb,\yb)^Tw+\nabla_xG(\xb,\yb)^T\sigma=0,\\
  \label{EqOptCond_b}&\nabla_y F(\xb,\yb)-\nabla_y\phi(\xb,\yb)^Tw+\nabla_yG(\xb,\yb)^T\sigma\\
  \nonumber&\hspace{4cm} -\nabla^2(\lb^Tg)(\yb)w+ 2\nabla^2(\eta^Tg)(\yb)\vb\in -(\F^v_1-\F^v_2)^\circ,\\
  \label{EqOptCond_c}&\vb^T\nabla^2 g(\yb)w\in (\F^\lambda_1-\F^\lambda_2)^\circ,\\
  \label{EqOptCond_d}&  \vb^T\nabla^2 g(\yb)\delta v\in T_{\Lb(\vb)}(\lb)^\circ, \F^\lambda_1=T_{\Lb(\vb)}(\lb)\cap [\vb^T\nabla^2 g(\yb)\delta v]^\perp,\\
  \label{EqOptCond_e}&\sigma_iG_i(\xb,\yb)=0,\ i=1,\ldots,p.
  \end{align}
  \end{subequations}
{Furthermore, if $\F^v_1-\F^v_2=\K_{\KbG}(\vb,\zba)$ and $\F^\lambda_1-\F^\lambda_2=T_{\Lb(\vb)}(\lb)$ then one of the following two cases must occur:  case (a) $\vb\not=0$;
  case (b)   $\vb=0$ and $\Lsp(\KbG)=\{0\}$ and   $\lb\in \Sigma(0,\zba)$.}\\
  Otherwise,  if $\F^v_1-\F^v_2\not=\K_{\KbG}(\vb,\zba)$ or $\F^\lambda_1-\F^\lambda_2\not=T_{\Lb(\vb)}(\lb)$
  then $\vb\not=0$ and there is some $\delta x\in\R^n$ such that
  \begin{subequations}
  \begin{align}
    &\nabla F(\xb,\yb){^T}(\delta x,\vb)=0, \label{EqOptCondIIa}\\
    &\nabla \phi(\xb,\yb)(\delta x,\vb)+\nabla^2(\lb^Tg)(\yb)\vb+\zba=0,\label{EqOptCondIIb}\\
    &\nabla G_i(\xb,\yb)^T (\delta x,\vb)\leq 0,\ \sigma_i\nabla G_i(\xb,\yb)^T (\delta x,\vb)=0,\ \forall i:G_i(\xb,\yb)=0,\label{EqOptCondIIc}
  \end{align}
  \end{subequations}
  and $T_{\Gr\widehat N_\Gamma}(\yb,\yba)$  is not locally polyhedral near $(\vb,-\nabla \phi(\xb,\yb)(\delta x,\vb))$.
\end{theorem}
\begin{proof}
  Consider $\delta x,\delta y, \delta\vb,\delta\vba, w^\ast, w, \delta a$ and $\sigma$ as in Lemma \ref{LemOptCond}. Then \eqref{EqOptCond_a} holds and
    \eqref{EqOptCond_e} follows from the observation that
  \[\sigma\in  N_{T_{T_{\R^m_-}(G(\xb,\yb))}\left(\nabla G(\xb,\yb)(\delta x,\delta y)\right)}(\delta a)\subseteq N_{T_{\R^m_-}(G(\xb,\yb))}\left(\nabla G(\xb,\yb)(\delta x,\delta y)\right)
  \subseteq N_{\R^m_-}(G(\xb,\yb)).\]

{\bf Case I: $\delta y=0$.} Then we also have $\delta x=0$ by \eqref{EqAuxOptCond_d} and thus
\[(\delta \vb,\delta \vba)\in T_{T_{\Gr \widehat N_\Gamma}(\yb,\yba)}\big(0,0\big)=T_{\Gr \widehat N_\Gamma}(\yb,\yba),\ (w^*,w)\in \widehat N_{T_{\Gr \widehat N_\Gamma}(\yb,\yba)}(\delta \vb,\delta \vba).\]

{\bf Subcase Ia:  $\delta \vb\not=0$.} Set $\vb =\delta \vb$  and by Lemma \ref{LemUnique_lambda} there are unique elements $\lb\in \Lb(\vb)$ and $\zba\in N_{\KbG}(\vb)$ such that $\delta\vba=\nabla^2(\lb^Tg)(\yb)\vb+\zba$. Since $$\tilde K(\vb,\delta\vba)=\K_{\KbG\times\Lb}(\vb,\lb,\zba,\vb^T\nabla^2g(\yb)\vb)=\K_{\KbG}(\vb,\zba)\times \K_{\Lb}(\lb,\vb^T\nabla^2g(\yb)\vb)$$ and $\K_{\Lb}(\lb,\vb^T\nabla^2g(\yb)\vb)=T_{\Lb(\vb)}(\lb)$, by (\ref{EqRegNormalConeTanConeGrNormalCone}) there is some $\eta\in T_{\Lb(\vb)}(\lb)$ such that
\begin{subequations}
\begin{align}
\label{EqOptCondAux_a}&w^*+\nabla^2(\lb^Tg)(\yb)w-2\nabla^2(\eta^T g)(\yb)\vb\in \K_{\KbG}(\vb,\zba)^\circ,\\
&w\in \K_{\KbG}(\vb,\zba),\\
\label{EqOptCondAux_c}&\vb^T\nabla^2g(\yb)w\in T_{\Lb(\vb)}(\lb)^\circ.
\end{align}
\end{subequations}
Set $\delta v=0$, $\F^v_2=\{0\}$, $\F^v_1= \K_{\KbG}(\vb,\zba)$, $\F^\lambda_2=\{0\}$, $\F^\lambda_1=T_{\Lb(\vb)}(\lb)$ implying $w\in \F^v_1-\F^v_2=\K_{\KbG}(\vb,\zba)$ and $\eta\in\F^\lambda_1-\F^\lambda_2=T_{\Lb(\vb)}(\lb)$. Then \eqref{EqOptCond_c} follows from \eqref{EqOptCondAux_c},  \eqref{EqOptCond_d} is fulfilled and \eqref{EqOptCond_b} follows from \eqref{EqAuxOptCond_c} and \eqref{EqOptCondAux_a}.

{\bf Subcase Ib:  $\delta \vb=0$.} By Lemma \ref{LemOptCond} the case $\delta\vb=0$ is only possible when $\Lsp(\KbG)=\{0\}$ and in this case there is some $\lb\in \Sigma(0,\delta \vba)$ such that \eqref{EqAuxOptCondSuppl} holds. It follows that the conditions of the theorem are fulfilled with $\vb=0$, $\zba=\delta\vba$, $\eta=0$, $\delta v=0$, $\F^v_2=\{0\}$, $\F^v_1= \K_{\KbG}(\vb,\zba)$, $\F^\lambda_2=\{0\}$, $\F^\lambda_1=T_{\Lb(\vb)}(\lb)$.

{\bf Case II: $\delta y\not=0$.} In this case set $\vb:=\delta y$, $\delta v:=\delta\vb$. Then $\big(\vb,-\nabla \phi(\xb,\yb)(\delta x,\vb)\big)\in T_{\Gr \widehat N_\Gamma}(\yb,\yba)$ and by Lemma \ref{LemUnique_lambda} there are unique $\lb\in \Lb(\vb)$ and $\zba\in N_{\KbG}(\vb)$ such that
  \[-\nabla \phi(\xb,\yb)(\delta x,\vb)=\nabla^2(\lb^Tg)(\yb)\vb+\zba.\]
In view of \eqref{EqBarMu1} and \eqref{EqBarMu2}, there are unique $\bar \mu\in T_{\Lb(\vb)}(\lb)$ and $\bar\zeta^*\in N_{\K_{\KbG}(\vb,\zba)}(\delta v)$ such that
  \begin{eqnarray*}
  && \delta\vba= \nabla^2(\lb^Tg)(\yb)\delta v +\nabla^2(\bar\mu^Tg)(\yb)\vb+\bar\zeta^*,\\
  &&  (\delta v, \bar\mu,\bar\zeta^*, 2\vb^T\nabla^2 g(\yb)\delta v)\in\Gr N_{\tilde K(\vb, -\nabla \phi(\xb,\yb)(\delta x,\vb))}.\end{eqnarray*}
  Further, by \eqref{EqRegNormalConeTanConeTanConeGrNormalCone} there is some $\eta$ such that
  \begin{align*}\big(w^\ast+&\nabla^2(\lb^Tg)(\yb)w- 2\nabla^2(\eta^Tg)(\yb)\vb, \vb^T\nabla^2 g(\yb)w, w, \eta\big)\\
\nonumber  & \in \tilde K(\delta y,-\nabla \phi(\xb,\yb)(\delta x,\vb),\delta v ,\delta\vba)^\circ\times \tilde K(\vb,-\nabla \phi(\xb,\yb)(\delta x,\vb),\delta v, \delta\vba).\end{align*}
By taking into account
\begin{align*}\tilde K(\vb,-\nabla \phi(\xb,\yb)(\delta x,\vb),\delta v,\delta\vba)&=\K_{\K_{\KbG}(\vb,\zba)\times T_{\Lb(\vb)}(\lb)}(\delta v,\bar \mu, \bar\zeta^*, 2\vb^T\nabla^2g(\yb)\delta v)\\
&=\K_{\K_{\KbG}(\vb,\zba)}(\delta v,\bar\zeta^*)\times \K_{ T_{\Lb(\vb)}(\lb)}(\bar \mu, 2\vb^T\nabla^2g(\yb)\delta v),\end{align*}
we obtain $\delta v\in \K_{\KbG}(\vb,\zba)$, $\bar \zeta^*\in N_{\K_{\KbG}(\vb,\zba)}(\delta v)$,
\begin{align*}(w^\ast+\nabla^2(\lb^Tg)(\yb)w- 2\nabla^2(\eta^Tg)(\yb)\vb, w)\in \K_{\K_{\KbG}(\vb,\zba)}(\delta v,\bar\zeta^*)^\circ\times \K_{\K_{\KbG}(\vb,\zba)}(\delta v,\bar\zeta^*),\nonumber \\
\nonumber(\vb^T\nabla^2 g(\yb)w,\eta)\in \K_{ T_{\Lb(\vb)}(\lb)}(\bar \mu, 2\vb\nabla^2g(\yb)\delta v)^\circ \times \K_{ T_{\Lb(\vb)}(\lb)}(\bar \mu, 2\vb^T\nabla^2g(\yb)\delta v).
\end{align*}
By defining $\F^v_1:=\K_{\KbG}(\vb,\zba)\cap [\bar\zeta^*]^\perp$, $\F^\lambda_1:=T_{\Lb(\vb)}(\lb)\cap [\vb^T\nabla^2 g(\yb)\delta v]^\perp$ and choosing $\F^v_2\subset\F^v_1$ and $\F^\lambda_2\subset \F^\lambda_1$ as those faces fulfilling $\delta v\in\ri \F^v_2$, $\bar\mu\in\ri\F^\lambda_2$ we obtain  $\K_{\K_{\KbG}(\vb,\zba)}(\delta v,\bar\zeta^*)= \F^v_1-\F^v_2$ and $\K_{ T_{\Lb(\vb)}(\lb)}(\bar \mu, 2\vb^T\nabla^2g(\yb)\delta v)=\F^\lambda_1-\F^\lambda_2$. Hence \eqref{EqOptCond} follows.
Since we have $\vb\not=0$, the claimed properties follow when $\F^v_1-\F^v_2=\K_{\KbG}(\vb,\zba)$ and $\F^\lambda_1-\F^\lambda_2=T_{\Lb(\vb)}(\lb)$. Otherwise, if $\F^v_1-\F^v_2 \not=\K_{\KbG}(\vb,\zba)$ and $\F^\lambda_1-\F^\lambda_2\not=T_{\Lb(\vb)}(\lb)$ the claimed properties follow as well.
\end{proof}

In the following remark we summarize some comments on the optimality conditions of Theorem \ref{ThOptCond}.
\begin{remark}\label{remark4}
\begin{enumerate}
\item If $\Lb(\vb)=\{\bar{\lambda}\}$ is a singleton then we have $\Lb(v)=\Lb(\vb)$ for all $v\in {\KbG}$ sufficiently close to $\vb$. Indeed,  by \cite[Lemma 3]{GfrOut16b} we have $\Lb(v)\subseteq\Lb(\vb)$ for every
$v\in {\KbG}$  sufficiently close to $\vb$ and $\Lb(v)\not=\emptyset$ for any $v\in {\KbG}$  by \cite[Proposition 4.3(iii)]{GfrMo15a}. As a consequence it follows from \cite[Proposition 3]{Gfr18} that $T_{\Gr\widehat N_\Gamma}(\yb,\yba)$  is locally polyhedral near $(\vb,\vba)$ for every $\vba$ satisfying $(\vb,\vba)\in T_{\Gr \widehat N_\Gamma}(\yb,\yba)$.  Thus by Theorem \ref{ThOptCond}  we must have  $\F_1^v-\F_2^v=\K_\KbG(\vb,\zba)$ and $\F_1^\lambda-\F_2^\lambda=T_{\Lb(\vb)}(\lb)=\{0\}$. Hence we have $\F_1^v=\K_{\KbG}(\vb,\zba)$, $\F_2^v=\Lsp(\K_{\KbG}(\vb,\zba))$, $\F_1^\lambda=\F_2^\lambda=\{0\}$, $\eta=0$ and hence   $\delta v\in \Lsp(\K_{\KbG}(\vb,\zba))$.
\if{\item
  If the set $T_{\Gr\widehat N_\Gamma}(\yb,\yba)$  is not locally polyhedral near $(\vb,-\nabla \phi(\xb,\yb)(\delta x,\vb))$, then by virtue of Proposition \ref{PropPolyhedr}, for every neighborhood $V$ of $\vb$ there is some $v\in V\cap \KbG$ such that $\Lb(v)\not=\Lb(\vb)$. In particular, $\Lb(\vb)$ cannot be a singleton because by \cite[Lemma 3]{GfrOut16b} we have $\Lb(v)\subseteq\Lb(\vb)$ for every $v$ sufficiently close to $\vb$.
\item If $\Lb(\vb)$ is a singleton, then $T_{\Lb(\vb)}(\lb)=\{0\}$ and consequently $\F_1^\lambda=\F_2^\lambda=\{0\}$ and $\eta=0$. Further we have $\F_1^v-\F_2^v=\KbG(\vb,\zba)$ and we can choose $\F_1^v=\K_{\KbG}(\vb,\zba)$, $\F_2^v=\{0\}$, $\delta v=0$.}\fi
\item If $\KbG$ is a subspace then for every $\vb\in\KbG$, $\zba\in N_{\KbG}(\vb)$ there holds $\K_{\KbG}(\vb,\zba)=\KbG$.

\item If $\K_{\KbG}(\vb,\zba)$ is a subspace then the only face of $\K_{\KbG}(\vb,\zba)$ is $\K_{\KbG}(\vb,\zba)$ itself and therefore $\F_1^v=\F_2^v=\F_1^v-\F_2^v=\K_{\KbG}(\vb,\zba)$. Similarly, if $\lb\in\ri \Lb(\vb)$ then $T_{\Lb(\vb)}(\lb)$ is a subspace and $\F_1^\lambda=\F_2^\lambda=\F_1^\lambda-\F_2^\lambda=T_{\Lb(\vb)}(\lb)$.
\end{enumerate}
\end{remark}
  \begin{example}[{cf. \cite[Examples 1,2]{GfrYe16a}}]\label{Ex1}
Consider the  MPEC
\begin{eqnarray*}
\min_{x,y} && F(x,y):=x_1-\frac{3}{2} y_1 + x_2-\frac 32 y_2- y_3 \nonumber \\
s.t. && 0\in \phi(x,y)+N_\Gamma( y), \\
&& G_1(x,y)=G_1(x):=-x_1-2x_2\leq 0,\nonumber \\
&& G_2(x,y)=G_2(x):=-2x_1-x_2\leq 0,\nonumber
\end{eqnarray*}
where
$$\phi(x,y):=\left(\begin{array}{c}
  y_1-x_1\\
  y_2-x_2\\
  -1\end{array}\right), \quad \Gamma:=\left\{y\in \R^3|
g_1(y):=y_3+\frac12 y_1^2\leq 0,\;g_2(y):=y_3+\frac12 y_2^2\leq 0  \right\}.$$
As it was demonstrated in \cite{GfrYe16a}, $\xb=(0,0)$ and  $\yb=(0,0,0)$ is the unique global solution and Assumption \ref{Ass1} is fulfilled. Straightforward calculations yield
\[\Lb=\{\lambda\in\R^2_+\mv \lambda_1+\lambda_2=1\},\ \KbG=\R^2\times\{0\}.\]
For every $v\in\R^3$, $\lambda\in\R^2$ we have $v^T\nabla^2(\lambda^Tg)(\yb)v=\lambda_1v_1^2+\lambda_2v_2^2$ yielding
\[\Lb(v)=\begin{cases}
  \{(1,0)\}&\mbox{if $\vert v_1\vert > \vert v_2\vert $,}\\
  \Lb&\mbox{if $\vert v_1\vert = \vert v_2\vert $,}\\
  \{(0,1)\}&\mbox{if $\vert v_1\vert < \vert v_2\vert $.}
\end{cases}\]
We now show that the mapping $g$ is 2-nondegenerate in every direction $0\not=v\in\KbG$, i.e. we have to verify
\begin{equation}\label{EqEx12NonDeg}\left(\begin{array}{c}\mu_1v_1\\\mu_2v_2\\0\end{array}\right)\in \big(N_{\KbG}(v))^+=\big(\{0\}\times\{0\}\times\R\big)^+=\{0\}\times\{0\}\times\R,\ \mu\in\big(\Lb(v)\big)^+\ \Rightarrow\ \mu=(0,0)\end{equation}
for every $0\not=v\in \KbG=\R^2\times\{0\}$.
If $\Lb(v)$ is a singleton this holds obviously true because then $\big(\Lb(v)\big)^+=\{0\}$. But the only case when $\Lb(v)$ is not a singleton is when $\vert v_1\vert=\vert v_2\vert$, which together with $v\not=0$ implies $\vert v_1\vert=\vert v_2\vert>0$ and we see that \eqref{EqEx12NonDeg} holds in this case as well.
We claim that the optimality conditions of Theorem \ref{ThOptCond} hold with $\vb =(1,1,0)$, $\zba=(0,0,0)$,  $\lb=(\frac 12,\frac 12)$, $\eta=(0,0)$, $\delta v=(0,0,0)$, $\F_1^v=\F_2^v =\KbG$, $w=-(1,1,0)$, $\F_1^\lambda=\F_2^\lambda=T_{\Lb(\vb)}(\lb)=\{(\eta_1,\eta_2)\mv \eta_1+\eta_2=0\}$ and $\sigma=(0,0)$. Indeed, we obviously have $\vb\in \KbG$, $\zba\in N_{\KbG}(\vb)$ and, since $\KbG$ is a subspace and $\lb\in\ri \Lb(\vb)$,  by Remark \ref{remark4}, $\K_{\KbG}(\vb, \zba)=\KbG$, $\F_1^v,\F_2^v$ are faces of $\K_{\KbG}(\vb, \zba)$,  $\F_1^\lambda,\F_2^\lambda$ are faces of $T_{\Lb(\vb)}(\lb)$ and $\delta v\in \ri \F_2^v=\KbG$,  $w\in \F^v_1-\F^v_2=\KbG$, $\eta\in \F^\lambda_1-\F^\lambda_2=T_{\Lb(\vb)}(\lb)$ and  $\sigma\geq 0$.
Conditions \eqref{EqOptCond_a}, \eqref{EqOptCond_b}, \eqref{EqOptCond_c} amount to
\begin{gather*}\left(\begin{array}{c}1\\1\end{array}\right)+\left(\begin{array}{c}w_1\\w_2\end{array}\right)-
\left(\begin{array}{c}\sigma_1+2\sigma_2\\2\sigma_1+\sigma_2\end{array}\right)
=\left(\begin{array}{c}0\\0\end{array}\right)\\
\left(\begin{array}{c}\frac 32\\\frac 32\\1\end{array}\right)+
\left(\begin{array}{c}w_1\\w_2 \\0\end{array}\right)+
\left(\begin{array}{c}\lb_1w_1\\\lb_2w_2 \\0\end{array}\right)-
2\left(\begin{array}{c}\eta_1\vb_1\\\eta_2\vb_2\\0\end{array}\right)\in (\F_1^v-\F_2^v)^\circ=\KbG^\perp=\{0\}\times\{0\}\times\R\\
\left(\begin{array}{c}\vb_1w_1\\\vb_2w_2 \end{array}\right)\in (F_1^\lambda-\F_2^\lambda)^\circ=T_{\Lb(\vb)}(\lb)^\perp=\R\left(\begin{array}{c}1\\1\end{array}\right)
\end{gather*}
and it is easy to see that they are fulfilled. Further, \eqref{EqOptCond_d} holds because of $\delta v=0$ and $\F_2^v$ is a subspace, and \eqref{EqOptCond_e} is fulfilled as well. Finally, we have $\F_1^v-\F^v_2=\K_{\KbG}(\vb,\zba)$, $\F_1^\lambda-\F_2^\lambda=T_{\Lb(\vb)}(\lb)$ and $\vb\not=0$. Thus the optimality conditions of Theorem \ref{ThOptCond} are fulfilled.
\end{example}
At the end of this section we want to formulate the necessary optimality conditions  in Theorem \ref{ThOptCond} in terms of index sets instead of faces. {Recall the definitions of $\bar I$, $\bar I(v)$, $\bar J^+(\lambda)$, $\bar J^+(\Xi)$ given in Subsection \ref{SubSecTwoNondegen}.}
\begin{theorem}\label{ThOptCondIndexSets} Assume that $(\xb,\yb)$ is a local minimizer for  problem \eqref{EqMPEC} fulfilling Assumption \ref{Ass1}. Further assume that $g$ is 2-nondegenerate in every nonzero critical  direction $0\not=v\in\KbG$ at $(\yb,\yba)$, where $\yba:=-\phi(\xb,\yb)$.
Then there are a critical direction $\vb\in\KbG$, a multiplier $\lb\in\Lb(\vb)$, index sets $\J^+$, $\J$, $\I^+$, and $\I$ with $\bar J^+(\lb)\subseteq \J^+\subseteq\J\subseteq \bar J^+(\Lb(\vb))\subseteq \bar J^+(\Lb)\subseteq \I^+\subseteq \I\subseteq \bar I(\vb)$ and elements $w\in\R^m$, $\eta,\xi\in\R^q$ and $\sigma\in \R^p_+$ such that
  \begin{subequations}
  \begin{align}
  \label{EqOptCondIndex_a}&0=\nabla_x F(\xb,\yb)-\nabla_x\phi(\xb,\yb)^Tw+\nabla_xG(\xb,\yb)^T\sigma,\\
  \label{EqOptCondIndex_b}&0=\nabla_y F(\xb,\yb)-\nabla_y\phi(\xb,\yb)^Tw +\nabla_yG(\xb,\yb)^T\sigma-\nabla^2(\lb^Tg)(\yb)w+\nabla g(\yb)^T\xi+2\nabla ^2(\eta^Tg)(\yb)\vb,\\
  \label{EqOptCondIndex_c}& \xi_i=0 \mbox{ if } i\not\in \I,\\
   \label{EqOptCondIndex_d}& \xi_i\geq 0, \nabla g_i(\yb)^T w\leq 0 \mbox{ if } i\in\I\setminus\I^+,\\
    \label{EqOptCondIndex_e}& \nabla g_i(\yb)^T w=0 \mbox{ if } i\in \I^+,\\
   \label{EqOptCondIndex_f}& \nabla g(\yb)^T\eta=0,\ \eta_i=0,i\not\in \J,\ \eta_i\geq 0, i\in \J\setminus \J^+,\\
  \label{EqOptCondIndex_g}&0=\sigma_iG_i(\xb,\yb),\ i=1,\ldots,p.
  \end{align}
Moreover, there are $\delta v\in\R^m$, $s_{\delta v},s_w\in\R^m$ and $\bar\mu\in\R^q$
such that
  \begin{align}
      \label{EqOptCondIndex_h}&\nabla g_i(\yb)^T \delta v=0,\ i\in \bar J^+(\Lb),\ \nabla g_i(\yb)^T \delta v\leq 0,\ i\in \bar I(\vb)\setminus \bar J^+(\Lb),\\
      \label{EqOptCondIndex_i}& \I=\{i\in \bar I(\vb)\mv \nabla g_i(\yb)^T \delta v= 0\},\\
      \label{EqOptCondIndex_j}&\nabla g_i(\yb)^T s_{\delta v}+\vb^T\nabla^2 g_i(\yb)\delta v=0,i\in \bar J^+(\lb),\ \nabla g_i(\yb)s_{\delta v}+\vb^T\nabla^2 g_i(\yb)^T \delta v\leq 0,i\in \bar J^+(\Lb(\vb))\setminus \bar J^+(\lb)\\
      \label{EqOptCondIndex_k}&\J=\{i\in \bar J^+(\Lb(\vb)))\mv \nabla g_i(\yb)^T s_{\delta v}+\vb^T\nabla^2 g_i(\yb)\delta v= 0\}\\
      \label{EqOptCondIndex_l}&\nabla g(\yb)^T\bar\mu=0,\ \bar\mu_i=0,i\not\in \J, \bar \mu_i\geq 0,i\in \J\setminus \bar J^+(\lb)\\
      \label{EqOptCondIndex_m}&\J^+=\bar J^+(\lb)\cup\{i\in \J\setminus \bar J^+(\lb)\mv \bar\mu_i>0\}\\
      \label{EqOptCondIndex_n}&\nabla g_i(\yb)^T s_w+\vb^T\nabla^2 g_i(\yb)w=0, i\in\J^+,\ \nabla g_i(\yb)^T s_w+\vb^T\nabla^2 g_i(\yb)w\leq 0, i\in\J\setminus\J^+.
  \end{align}
Furthermore, if $\I=\bar I(\vb)$,  $\J^+=\bar J^+(\lb)$ and  $\J=\bar J^+(\Lb(\vb))$ then one of the following two cases must occur: case (a)  $\vb\not=0$;
case (b) $\vb=0$, $\Lsp(\KbG)=\{0\}$ and $\lb\in\Sigma(0,\zba)$ for some $\zba=\sum_{i\in \I}\nabla g_i(\yb)\alpha_i$ with $\alpha_i>0$, $i\in \I^+\setminus \bar J^+(\Lb)$.

 Otherwise, if either $\I\not=\bar I(\vb)$ or $\J^+\not=\bar J^+(\lb)$ or $\J\not=\bar J^+(\Lb(\vb))$ then $\vb\not=0$ and there are some $\delta x\in\R^n$ and some $\zba=\sum_{i\in \I}\nabla g_i(\yb)\alpha_i$ with $\alpha_i>0$, $i\in \I^+\setminus \bar J^+(\Lb)$ such that conditions \eqref{EqOptCondIIa}-\eqref{EqOptCondIIc} hold and
  $T_{\Gr \widehat N_\Gamma}(\yb,\yba)$ is not locally polyhedral near $(\vb,-\nabla \phi(\xb,\yb)(\delta x,\vb))$.
\end{subequations}
\end{theorem}
\begin{proof}
  Let $\vb,\lb,\delta v,\eta,w,\sigma, \F_1^v,\F_2^v,\F_1^\lambda,\F_2^\lambda$ as in Theorem \ref{ThOptCond}. The index sets $\I,\I^+, \J,\J^+$ were chosen such that
  \begin{align*}&\F_1^v-\F_2^v=\{s\mv \nabla g_i(\yb)^T s=0,i\in \I^+,\ \nabla g_i(\yb)^T s\leq0,i\in\I\setminus\I^+\},\\
      &F_1^\lambda-\F_2^\lambda=\{\mu\mv\nabla g(\yb)^T\mu=0,\ \mu_i=0,i\not\in \J,\ \mu_i\geq 0,i\in \J\setminus\J^+ \}.
  \end{align*}
  Then $(\F_1^v-\F_2^v)^\circ=\{\nabla g(\yb)^T\xi\mv \xi_i=0,i\not\in\I,\ \xi_i\geq 0,i\in\I\setminus\I^+\}$ and conditions \eqref{EqOptCondIndex_a}-\eqref{EqOptCondIndex_g} follow immediately. \eqref{EqOptCondIndex_h} states that $\delta v\in T_{\KbG}(\vb)$ whereas \eqref{EqOptCondIndex_i} results from the requirement $\delta v\in\ri \F_2^v$ together with Proposition \ref{PropBasicPropCone}. The index set $\I^+$ is related with $\zba\in N_{\KbG}(\vb)$. Since we do not have any further condition on $\zba$, the same applies to $\I^+$. \eqref{EqOptCondIndex_j} states that $\vb^T\nabla^2g(\yb)\delta v\in T_{\Lb(\vb)}(\lb)^\circ$ and condition \eqref{EqOptCondIndex_k} results from the second part of \eqref{EqOptCond_d}.
  The point $\bar\mu$  denotes any point in $\ri F_2^\lambda$ yielding the condition \eqref{EqOptCondIndex_m} by Proposition \ref{PropBasicPropCone}. Finally, condition \eqref{EqOptCondIndex_n} is equivalent to \eqref{EqOptCond_c}.
\end{proof}
{The optimality conditions of Theorem \ref{ThOptCondIndexSets} dramatically simplify under the assumption that $\Lb=\{\lb\}$ is a singleton. In this case  the condition that $g$ is 2-nondegenerate in every nonzero critical  direction $0\not=v\in\KbG$ at $(\yb,\yba)$ holds automatically. Further, { by the formula for the tangent cone in Theorem \ref{ThTanConeGrNormalCone}, it is easy to see that in this case the tangent cone} $T_{\Gr \widehat N_\Gamma}(\yb,\yba)$ is polyhedral and therefore by Theorem \ref{ThOptCondIndexSets}, we must have
$$\bar J^+(\bar\lambda)=\J^+=\J =\bar J^+(\Lb)\subseteq \I^+\subseteq \I= \bar I(\vb).$$
Let $w\in\R^m$, $\xi\in\R^q$ and $\sigma\in \R^p_+$ be those found in Theorem \ref{ThOptCondIndexSets} and take, as already pointed out in Remark \ref{remark4}, $\eta=0$.  Moreover, we can take $\delta v=0$, $s_{\delta v}=0$ and $\bar\mu=0$ in order to fulfill conditions \eqref{EqOptCondIndex_h}-\eqref{EqOptCondIndex_m}. Finally, the assumption that $\Lb=\{\bar \lambda\}$ is a singleton implies that the gradients $\nabla g_i(\yb)$, $i\in \bar J^+(\bar\lambda)$ are linearly independent and therefore there always exists an element $s_w$ fulfilling \eqref{EqOptCondIndex_n}. Therefore we have the following corollary.
\begin{corollary}\label{CorUniquemultiplier} Assume that $(\xb,\yb)$ is a local minimizer for problem \eqref{EqMPEC} fulfilling Assumption \ref{Ass1}. Further assume that the multiplier set $\Lb=\{\bar \lambda\}$ is a singleton.
Then there are a critical direction $\vb\in\KbG$,  index set $\I^+$ with $\bar J^+(\lb)\subseteq \I^+\subseteq  \bar I(\vb)$ and elements $w\in\R^m$, $\xi\in\R^q$ and $\sigma\in \R^p_+$ such that
  \begin{subequations}
  \label{newEqOptCondIndex}
  \begin{align}
  \label{newEqOptCondIndex_a}&0=\nabla_x F(\xb,\yb)-\nabla_x\phi(\xb,\yb)^Tw+\nabla_xG(\xb,\yb)^T\sigma,\\
  \label{newEqOptCondIndex_b}&0=\nabla_y F(\xb,\yb)-\nabla_y\phi(\xb,\yb)^Tw +\nabla_yG(\xb,\yb)^T\sigma-\nabla^2(\lb^Tg)(\yb)w+\nabla g(\yb)^T\xi,\\
  \label{newEqOptCondIndex_c}& \xi_i=0 \mbox{ if } i\not\in  \bar I(\vb),\\
   \label{newEqOptCondIndex_d}& \xi_i\geq 0, \nabla g_i(\yb)^T w\leq 0 \mbox{ if } i\in \bar I(\vb)\setminus\I^+,\\
    \label{newEqOptCondIndex_e}& \nabla g_i(\yb)^T w=0 \mbox{ if } i\in \I^+,\\
  \label{newEqOptCondIndex_g}&0=\sigma_iG_i(\xb,\yb),\ i=1,\ldots,p.
  \end{align}
\end{subequations}
\end{corollary}
}
{We now want to compare our optimality conditions for (MPEC) with the known M-stationarity conditions for (MPCC)  defined as follows.
\begin{definition}[M-stationary condition for (MPCC)]
  \label{DefOptCondMPCC}  Let $(\xb,\yb, \lb)$ be a feasible solution for  problem (MPCC).
We say that $(\xb,\yb, \lb)$ satisfies the M-stationary condition for (MPCC) if there exist $w\in \R^m, \xi\in \R^q, \sigma\in\R^p_+$ such that
  \begin{subequations}
  \label{EqOptCond-MPCC}
  \begin{align}
  \label{EqOptCond-MPCC_a}&0=\nabla_x F(\xb,\yb)-\nabla_x\phi(\xb,\yb)^Tw+\nabla_xG(\xb,\yb)^T\sigma,\\
  \label{EqOptCond-MPCC_b}&0=\nabla_y F(\xb,\yb)-\nabla_y\phi(\xb,\yb)^Tw+\nabla_yG(\xb,\yb)^T\sigma-\nabla^2(\lb^Tg)(\yb)w+\nabla
   g(\yb)^T\xi,\\
   \label{EqOptCond-MPCC_c}& \xi_i=0 \mbox{ if } g_i(\yb)<0, \lb_i=0,\\
    \label{EqOptCond-MPCC_d}& \nabla g_i(\yb)^T w=0 \mbox{ if } g_i(\yb)=0, \lb_i >0,\\
    \label{EqOptCond-MPCC_f}& \mbox{either } \xi_i>0,   \nabla g_i(\yb)^T w<0 \mbox{ or } \xi_i\nabla g_i(\yb)^T w=0 \mbox{ if } g_i(\yb)=\lb_i=0,\\
  \label{EqOptCond-MPCC_e}&0=\sigma_iG_i(\xb,\yb),\ i=1,\ldots,p.
  \end{align}
  \end{subequations}
\end{definition}
It is well know that  problem (MPCC) may  not be equivalent to  problem (MPEC) in the case when the lower level problem does not have unique multiplier. Moreover, in this case it is also possible that at a locally optimal solution $(\xb,\yb,\lb)$ even the weakest known constraint qualification, the MPCC-GCQ (Guignard constraint qualification), ensuring M-stationarity is not fulfilled. E.g., it was shown in \cite{GfrYe16a} that for  Example \ref{Ex1} MPCC-GCQ does not hold at $(\xb,\yb,\lambda)$ for any $\lambda\in\Lb$.}

{ For the case when the multiplier set $\Lb=\{\bar \lambda\}$ is a singleton, we now  compare our  necessary optimality conditions of Corollary \ref{CorUniquemultiplier} with M-stationarity condition for (MPCC).  By \cite[Proposition 2]{Ad-Hen-Out}, the assumption of MSCQ for (MPEC) is weaker than the corresponding one for (MPCC).}
Suppose that $(\xb,\yb, \bar \lambda )$ satisfies the optimality condition in Corollary \ref{CorUniquemultiplier} and let $w\in\R^m$, $\xi\in\R^q$ and $\sigma\in \R^p_+$ be those found in Corollary \ref{CorUniquemultiplier}. Then (\ref{EqOptCond-MPCC_a})-(\ref{EqOptCond-MPCC_b}) and (\ref{EqOptCond-MPCC_e})  hold.
Since
$$i\not \in  \bar I(\vb) \Longleftrightarrow  \begin{array}{ll}
\mbox{ either } g_i(\bar y)=0, \nabla g_i(\bar y)^T \bar v<0, \bar{\lambda}_i=0\\
\mbox{ or }  g_i(\bar y)<0, \bar{\lambda}_i=0 \end{array}$$
and $\bar J^+(\bar\lambda)=\{ i| g_i(\bar y)=0, \bar \lambda_i>0\}$,
\eqref{newEqOptCondIndex_c} and \eqref{newEqOptCondIndex_e}
implies that   $\xi_i =0$ if $\bar{\lambda}_i=0$ and
$\nabla  g_i(\bar y) ^T w=0 $ if $ \bar \lambda_i >0$. It follows that \eqref{EqOptCond-MPCC_c}-\eqref{EqOptCond-MPCC_f} hold.
Therefore $(\xb,\yb, \bar \lambda )$ must satisfy the M-stationary condition for (MPCC) as well.

It is not difficult to show that the M-stationarity conditions of
Definition \ref{DefOptCondMPCC} imply the necessary optimality conditions of
Corollary \ref{CorUniquemultiplier}
provided the {\em linear independence constraint qualification} (LICQ) holds for the lower level problem at $\bar y$.  Indeed, under LICQ  the multiplier set $\Lb=\{\bar \lambda\}$ is a singleton. Given $w$, $\xi$ and $\sigma$ fulfilling \eqref{EqOptCond-MPCC}, define
$$
\I^+:=\bar J^+(\bar\lambda)\cup\{i\in\bar I\mv \xi_i<0\},\ \I:=\I^+\cup\{i\in\bar I\mv \nabla g_i(\yb)^T w\leq 0,\xi_i>0\}$$
and then find $\vb$ fulfilling
\[\nabla g_i(\yb)^T \vb=0, i\in \I, \ \nabla g_i(\yb)^T \vb=-1,i\in \bar I\setminus\I\]
which exists due to the imposed LICQ. It follows that $\bar{v} \in \KbG$, $\I=\bar I(\vb)$  and that the conditions \eqref{newEqOptCondIndex} are fulfilled. {Hence, under LICQ the optimality conditions of Corollary \ref{CorUniquemultiplier} are equivalent to the M-stationarity conditions of Definition \ref{DefOptCondMPCC}.}
However, the following example demonstrates that the optimality conditions of Corollary \ref{CorUniquemultiplier}
are sharper, when $\Lb$ is a singleton but LICQ fails.
\begin{example}\label{Ex2}
  Consider the problem
  \begin{align*} \min_{x\in\R,y\in\R^2}&x+y_1+y_2\\
   \mbox{subject to}&\ 0\in \left(\begin{array}{c}x+2y_1\\x+y_2\end{array}\right)+\widehat N_\Gamma(y)\ \mbox{where}\ \Gamma:=\{y\in\R^2\mv y_1\leq 0,\ -y_2\leq 0, y_1+y_2\leq 0\}\end{align*}
  at $\xb=0$, $\yb=(0,0)$. Straightforward calculations yield that $\Lb=\{(0,0,0)\}$ and that $(\xb,\yb)$ is not a local minimizer. However, the M-stationary conditions of Definition \ref{DefOptCondMPCC} amount to
  \begin{align*}
    &0=1-(w_1+w_2),\\
    &\left(\begin{array}{c}0\\0\end{array}\right)=\left(\begin{array}{c}1\\1\end{array}\right)-\left(\begin{array}{c}2w_1\\w_2\end{array}\right)
    +\left(\begin{array}{c}\xi_1+\xi_3\\-\xi_2+\xi_3\end{array}\right),\\
    &\big(w_1<0,\ \xi_1>0\big)\vee \big (w_1\xi_1=0 \big ),\\
    &\big(-w_2<0,\ \xi_2>0\big)\vee \big (-w_2\xi_2=0\big),\\
    &\big(w_1+w_2<0,\ \xi_3>0\big)\vee \big ( (w_1+w_2)\xi_3=0 \big ),
  \end{align*}
  where $(X) \vee (Y) $ denotes either       $X$ or $Y$ holds,
  and are uniquely fulfilled with $w=(0,1)$ and $\xi=({-1},0,0)$. Now let us show that the optimality condition in Corollary \ref{CorUniquemultiplier} with $w=(0,1)$, and $\xi=({-1},0,0)$\if{,  $s_{\delta v}=0, s_w=0$ and $\bar\mu=0$}\fi does not hold.

By applying Theorem \ref{ThSuffCondMS_FOGE} we deduce that MSCQ and consequently Assumption \ref{Ass1} are fulfilled.
      Since the only multiplier is $\bar{\lambda}=(0,0,0)$, we have $\bar J^+(\bar\lambda)=\emptyset$ and $\KbG=\Gamma$.  Since $w_1=0$ and $\nabla g_1(\bar y)w=w_1$, (\ref{newEqOptCondIndex_e}) holds if and only if $\{1\}\subseteq \I^+$. Since $\{1\} \subseteq \I^+ \subseteq  \bar I(\vb)$,    one must have $\nabla g_1(\bar y)^T  \bar{v}=\bar{v}_1=0$.
   But then (\ref{newEqOptCondIndex_c}) means $\xi_3=0$ if $3\not \in \bar I (\vb)$, which in turn means that $\nabla g_3(\bar y)^T \bar{v}=\bar{v}_1+\bar{v}_2<0$.
   This is impossible since we cannot find $\vb\in \KbG=\Gamma$ satisfying $\vb_1=0$ and $\vb_1+\vb_2<0$.
   This shows   that the M-stationarity conditions do not correctly describe the faces of the critical cone.
\end{example}

Finally we want to compare our results with the ones of Gfrerer and Outrata \cite{GfrOut16a}, where the limiting normal cone of the normal cone mapping was computed and thus could be used to compute the conventional M-stationarity conditions for problem
(MPEC). The assumption 2-LICQ used in \cite{GfrOut16a} cannot be characterized by first-order and second-order derivatives of the constraint mapping $g$, however the sufficient condition for 2-LICQ as stated in \cite[Proposition 3]{GfrOut16a} is stronger than the 2-nondegeneracy  assumption we use.
The sufficient condition for 2-LICQ  in direction $\vb\in \KbG$ used in \cite[Proposition 3]{GfrOut16a} now states that for every index set $J$ with $\bar J^+(\Lb(\vb))\subseteq J\subseteq \bar I(\vb)$ satisfying
\[\nabla g_i(\yb)^T s+\vb^T\nabla^2 g_i(\yb)\vb\begin{cases}
  =0&i\in J\\ \leq0&i\in\bar I(\vb)\setminus J\end{cases}\ \mbox{for some $s\in\R^m$}\]
the mapping $(g_i)_{i\in J}$ is 2-regular in direction $\vb$. Such an index set $J$ always exists, e.g. by duality theory of linear programming $J=\bar J^+(\Lb(\vb))$ is a possible choice.  Now choose $J$ large enough such that for every $j\in \bar I(\vb)\setminus J$ the gradient $\nabla g_j(\yb)$ linearly depend on $\nabla g_i(\yb)$, $i\in J$. so that $J$ meets the requirements on the index set $\hat J$ used in Subsection \ref{SubSecTwoNondegen} and we see that the assumption of 2-regularity of $(g_i)_{i\in J}$ in direction $\vb$ implies 2-nondegeneracy of $g$ in direction $\vb$.

Further one can show that the necessary conditions of Theorem \ref{ThOptCondIndexSets} are stronger than the M-stationary conditions which one could obtain with the M-stationary conditions of \cite[Theorem 4]{GfrOut16a} insofar as an additional condition on $\delta x$ is included in Theorem \ref{ThOptCondIndexSets}.

\end{document}